\documentclass[12pt]{article}

\usepackage{a4wide}
\usepackage{amsmath}
\usepackage{amsthm}
\usepackage{array}
\usepackage{amssymb}
\usepackage{amsfonts}
\usepackage[english]{babel}
\usepackage{epsf}
\usepackage{epsfig}
\usepackage{graphicx}

\usepackage{xcolor}

\usepackage{calligra}

\newcommand{\R}{\mathbb{R}}

\newcommand{\Z}{\mathbb{Z}}

\newcommand{\Sd}{\textup{Sd}}

\newcommand{\Lk}{\textup{Lk}}

\DeclareMathAlphabet{\mathcalligra}{T1}{calligra}{m}{n}

\newtheorem{theorem}{Theorem}[section]
\newtheorem{lemma}[theorem]{Lemma}
\newtheorem{proposition}[theorem]{Proposition}
\newtheorem{corollary}[theorem]{Corollary}
\newtheorem{remark}[theorem]{Remark}

\newtheorem{definition}[theorem]{Definition}

\begin{document}

\title{Morse shellings and compatible discrete Morse functions}

\author{Nerm{\accent95\i}n Salepc{\accent95\i} and Jean-Yves Welschinger}
\maketitle

\begin{abstract}
We introduce  a notion of Morse shellings (and tilings) on finite simplicial complexes which extends the classical one and  its relation to discrete  Morse theory.
Skeletons and barycentric subdivisions of Morse shellable (or tileable) simplicial complexes are Morse shellable (or tileable). Moreover, every triangulated closed surface is Morse shellable while every closed three-manifold carries Morse shellable triangulations. Finally, any shelling encodes a class of discrete Morse functions whose critical points are in one-to-one correspondence, preserving the index, with the critical tiles of the shelling. 


\end{abstract}

\section{Introduction}

We recently \cite{SW3} introduced a notion of tilings of finite simplicial complexes which are partitions of their geometric realizations by tiles. A tile is a maximal simplex deprived of several facets, that is of several codimension one faces. In each dimension $n$, there are thus $n+2$ different tiles up to affine isomorphism, denoted by $T_0^n,\ldots, T^n_{n+1}$ depending on the number of facets that have been removed, and the closed simplex itself is one of them, namely $T^n_0$, while the open simplex is another one, namely $T^n_{n+1}$. Not all simplicial complexes are tileable, but skeletons and barycentric subdivisions of tileable ones  are tileable by Theorem~1.9 of \cite{SW3}. The boundary of any convex simplicial polytope is tileable, even shellable, by \cite{BruMan}, and the product of any sphere with a torus of positive dimension carries tileable triangulations which cannot be shelled by \cite{Welsch}, see \S \ref{SSect_Shelling}.
These tilings provide a geometric way to understand the $h$-vectors of finite tileable simplicial complexes, see \cite{Ful, Stan, Z} for a definition. Namely, if $h^n_k$ denotes the number of tiles $T^n_k$ needed to tile a complex  $K$, then $(h^n_0,\ldots, h^n_{n+1})$ coincides with the $h$-vector of $K$ provided  that $h^n_0=1$ and in general, two tilings of $K$ have the same $h$-vector  $(h^n_0,\ldots, h^n_{n+1})$ provided they have the same number of tiles $T^n_0$, see Theorem~1.8 of \cite{SW3}.
These tilings also appeared to be useful to produce packings by disjoint simplices of the successive barycentric subdivisions $\Sd^d(K), d> 0$, see \S~5 of \cite{SW3}. They actually moreover seemed to be closely related to the discrete Morse theory of Robin Forman~\cite{F1} even though this aspect has not been investigated in \cite{SW3}. The tiles $T^n_0$ behaved as critical points of index zero, the tiles $T^n_{n+1}$ as critical points of index $n$ of a Morse function  and the other ones as regular points. No analog though of critical points of intermediate indices. We now fill this gap. 

We define a Morse tile to be a closed simplex deprived of several facets together with a unique face of possibly higher codimension. It is critical if and only if this codimension is maximal, see Definition \ref{Defn_Morse}. A Morse tiling of a finite simplicial complex $K$ is a partition by Morse tiles such that for every $j\geq 0$, the union of tiles of dimension greater than $j$ is a simplicial subcomplex, see Definition \ref{Defn_Morse tiling}. 
The previous notion of tiling, due to its relation with $h$-vectors, is now called $h$-tiling and slightly generalized to allow for  tiles of various dimensions, see Definition \ref{Defn_$h$tiling}. 
A Morse tiling on $K$ is moreover said to be shellable iff $K$ admits a filtration $\emptyset=K_0 \subset K_1\subset\ldots\subset K_N=K$ by simplicial subcomplexes such that for every $i\in \{1,\ldots, N\},$ $K_i\setminus K_{i-1}$ consists of a single Morse tile, see Definition \ref{Morseshellable}. Replacing Morse tiles by basic tiles in this definition, we recover the classical notion of shellability, see Lemma \ref{Lemmashel}, without some non-emptiness assumption though, see Remark \ref{Rem_rem0} and \cite{Koz, Z} for instance. These definitions actually extend to a larger class of sets, the Morse tileable or shellable sets, see subsections \ref{Ssect_MorseTiling} and \ref{SSect_Shelling}.
We prove the following tiling theorem, see Corollary \ref{Cor_SkeletonM}, Corollary \ref{Cor_BarSubM} and Theorem \ref{Thm_SkelBarSubShell}. 
\begin{theorem}\label{Thm_intro0}
Skeletons and barycentric subdivisions of Morse tileable (resp. shellable) sets are Morse tileable (resp. shellable). Moreover, every Morse tiling on such a set induces Morse tilings on its barycentric subdivisions containing the same number of critical tiles with the same indices. 
\end{theorem}

This tiling theorem  relies in particular on the fact that the first barycentric subdivision of a Morse tile is itself a disjoint union of Morse tiles, see Theorem \ref{Thm_Sd}.
Given a Morse tiling on a finite simplicial complex $K$, we also deduce packings by disjoint simplices in its successive barycentric subdivisions, see  Proposition \ref{Proppack}. Such packings were used in \cite{SW3} 
to improve upper estimates on the expected Betti numbers of random subcomplexes. 

We then associate to every Morse tiling  a set of discrete vector fields, in the sense of Robin Forman \cite{F1, F2}, which are compatible with the tiling. Their critical points are in one-to-one  correspondence with the critical Morse tiles, see subsection \ref{SSect_VF}. Moreover, due to Theorem~9.3 of \cite{F1}, these vector fields are gradient vector fields of discrete Morse functions provided they have no non-stationary  closed paths, see subsection \ref{subsecdiscrete}.  We provide a criterium for the latter condition to be satisfied, Theorem \ref{Thm_DescreteMorseFiltration}, that applies to  Morse shellable complexes. 
 
We thus get, see Corollary \ref{Cor_Mshellable}.
\begin{theorem}\label{Thm_intro0B}
Any discrete vector field compatible with a Morse shelling on a finite simplicial complex is the  gradient vector field of a discrete self-indexing Morse function whose critical points are in one-to-one correspondence  with the critical tiles of the shelling, preserving the index.
\end{theorem}
Such a relation with discrete Morse theory is well known in the case of shellings, see for instance \cite{Cha}. However,
this result applies to  all triangulations of closed manifolds in dimension one and two. Indeed, 
 
\begin{theorem}\label{Thm_intro1}
Every closed triangulated manifold of dimension less than three is Morse shellable.

\end{theorem}
Moreover it  applies to all closed three-manifolds for some  of their triangulations.
 
\begin{theorem} \label{Thm_intro2}
Every closed manifold of dimension less than four carries a Morse shellable triangulation. Moreover given any smooth Morse function on the manifold, there exists a Morse shelled triangulation whose number of critical tiles coincides with the number of critical points of the function, in any index.  \end{theorem}

Not every Morse tiling shares the property given by Theorem \ref{Thm_intro0B} and  not every simplicial complex is Morse tileable, see subsection \ref{SSect_Remarks}. It would be of interest to find a triangulation of a closed manifold which is not Morse tileable. By \cite{Welsch3}, every triangulation becomes Morse shellable after finitely many stellar subdivisions at maximal simplices.
Given  a Morse  tiling on a closed triangulated manifold, there are many different compatible discrete vector fields  and thus many associated discrete Morse functions in the case of Theorem \ref{Thm_intro1} and Theorem \ref{Thm_intro2},  but they all have the same number of critical points with same indices. These critical points are in one-to-one correspondence with the critical tiles of the tiling, preserving the index. Such a Morse tiling thus provides an efficient way to bound the topology of the manifold. 

\begin{corollary}\label{Cor_intro} Let $X$ be a closed triangulated $n$-manifold equipped with a Morse tiling $\mathcal{T}$ admitting a compatible discrete Morse function. Then, each Betti number $b_k(X)$ of $X$ is bounded from above by the number of critical tiles $c_k(\mathcal{T})$ of index $k$ of $\mathcal{T}$ and the Morse inequalities hold true, namely $\sum_{i=0}^k(-1)^{k-i}b_i(X)\leq \sum_{i=0}^k(-1)^{k-i}c_i(\mathcal{T}) $  for every $0\leq k\leq n$ with equality if $k=n.$
\end{corollary}
A Morse shelling on a finite simplicial complex induces in fact two spectral sequences which converge to the (co)homology of the complex and whose first pages are free modules over the critical tiles of the shelling, see \cite{Welsch2}.
It would be of interest to extend Theorem \ref{Thm_intro2} to all dimensions. We also actually do not know which are the closed three-manifolds that admit  $h$-tileable triangulations, see subsection \ref{SSect_Remarks}. 

The second section of this paper is devoted to Morse tiles and tilings with the proof of Theorem \ref{Thm_intro0} while the third one is devoted to discrete Morse theory and the proofs of Theorem \ref{Thm_intro0B},  Theorem \ref{Thm_intro1}, Theorem \ref{Thm_intro2} and Corollary \ref{Cor_intro}.\\

\textbf{Acknowledgement:}
The second author is partially supported by the ANR project MICROLOCAL (ANR-15CE40-0007-01).

\section{Morse tilings}

\subsection{Morse tiles}\label{Ssect_htiling}

Let us recall that an $n$-simplex is the convex hull of $n+1$ points affinely independent in some real affine space and that the standard $n$-simplex $\Delta_n$ is the one spanned by the standard affine basis of $\R^{n+1}$, see~\cite{M}. These are all isomorphic one to another by some affine isomorphism. A face of a simplex is the convex hull of a subset of its vertices.

For every $n>0$ and every $k\in \{0,\ldots,n+1\}$, we set $T_k^n=\Delta_n\setminus (\sigma_1\cup\ldots\cup\sigma_k),$ where $(\sigma_i)_{ i\in \{1,\ldots,n+1\}}$ denote the facets of $\Delta_n$, that is its codimension one faces.  In particular, the tile $T_{n+1}^n$ is the open $n$-simplex $\stackrel{\circ}{\Delta}_n$ and $T_0^n = \Delta_n$ is the closed one.
These standard tiles were introduced in \cite{SW3} and one of their key properties is the following.

\begin{proposition}[Proposition~4.1 of \cite{SW3}] \label{Prop_4.1}
For every $n>0$ and every $k\in \{0,\ldots, n+1\},$ $T^{n+1}_k$ is a cone over $T_k^n$, deprived of  its apex if $k\neq 0$. Moreover, 
$T^{n+1}_k$ is a disjoint union  $T_{n+2}^{n+1}\sqcup T_k^n\sqcup T_{k+1}^n\sqcup\ldots \sqcup T^n_{n+1}.$
In particular,  the cone $T_k^{n+1}$ deprived of its base $T^n_k$ is $T_{k+1}^{n+1}$.
\end{proposition}

\begin{proof}
If $k=0,$ $T_{0}^{n+1}=\Delta_{n+1}=c\ast T_{0}^n$, where $c$ denotes a vertex of $\Delta_{n+1}$. If $k>0$, $T_k^n=\Delta_n\setminus (\sigma_1\cup\ldots\cup\sigma_{k})$ by definition, where $\sigma_i$ is a facet for every $i\in\{1,\ldots, k\},$ and so  $(c\ast T^n_k)\setminus \{c\}=(c\ast \Delta_n)\setminus ((c\ast \sigma_1)\cup\ldots\cup (c\ast \sigma_{k})).$ However, $c\ast \Delta_n=\Delta_{n+1}$ and $\theta_i=c\ast \sigma_i$ is an $n$-simplex,  $i\in \{1,\ldots, k\}$. It follows from the definition that $T_{k}^{n+1}=(c\ast T_k^n)\setminus \{c\}.$ It is the cone over $T_k^n$ deprived of its apex $c$. The base $T_k^n$  of this cone is the intersection of $T_{k+1}^{n+1}$ with the base $\theta=\Delta_n$ of the cone $\Delta_{n+1}=c\ast \Delta_n$. Thus, $T_{k}^{n+1}\setminus T_{k}^{n}=(c\ast \Delta_n)\setminus (\theta_1\cup\ldots\cup \theta_{k}\cup \theta)=T^{n+1}_{k+1}$.  The result holds true for $k=0$ as well, since by definition $T_{1}^{n+1}=\Delta_{n+1}\setminus \Delta_n=T_{0}^{n+1}\setminus T^n_{0}.$
By induction, we deduce that for every $k\in \{0,\ldots,n+1\},$ $T_{k}^{n+1}\cap \partial \Delta_{n+1}$ is the disjoint union $T_k^n\sqcup\ldots\sqcup T^n_{n+1}.$ 
\end{proof} 

\begin{definition}\label{Defn_Basic}
 A basic tile is a subset of a simplex isomorphic to a standard tile $T^n_k$ via some affine isomorphism. The integer $n\geq0$ is the dimension of the tile while $k\in\{0,\ldots,n+1\}$ is the order of the tile.
\end{definition}

The $j$-skeleton  of a tile $T_k^n$ is by definition the intersection of the $j$-skeleton of $\Delta_n$ with $T_k^n\subset \Delta_n$. Proposition \ref{Prop_4.1} provides a partition of the $(n-1)$-skeleton of $T^n_k$ by basic tiles and by induction  it provides a partition of all its skeletons by basic tiles of the corresponding dimensions.

\begin{proposition}\label{Prop_UniqueTile}
For every $n\geq 0$, every $0\leq k\leq n+1$ and every $j\in \{k-1,\ldots, n\},$ any partition of the $j$-skeleton  of $T^n_k$ given by Proposition \ref{Prop_4.1} contains only tiles of order $\geq k.$ Moreover, it contains  a unique tile of order $k$ which is the trace of a $j$-dimensional face of $\Delta_n$ with   $T^n_k$. If $j<k-1,$ the $j$-skeleton of $T^n_k$ is empty.
\end{proposition}

\begin{proof}
This result is given by Proposition \ref{Prop_4.1} when $j=n-1$  and $T^{n-1}_k$ is indeed the trace of a facet of $\Delta_n$ with $T^n_k$, since it is the intersection of the subcomplex $T^{n-1}_0\sqcup\ldots\sqcup T^{n-1}_k$ of $\Delta_n$ with $T^n_k$, while $T^{n-1}_0\sqcup\ldots\sqcup T^{n-1}_{k-1}$ is disjoint from $T^n_k$. The result then follows from Proposition \ref{Prop_4.1} and a decreasing induction in general.
\end{proof}

Let now $\tau$ be a face of $\Delta_n$ not contained in $\sigma_1\cup\ldots\cup\sigma_k$ and let $l$ be its dimension, so that $k \leq l+1 \leq n$. We set $T_k^{n,l}=\Delta_n\setminus (\sigma_1\cup\ldots\cup\sigma_k\cup\tau)$, it is uniquely defined by $k,l,n$ up to permutation of the vertices of $\Delta_n$.

\begin{definition} \label{Defn_Morse}
A Morse tile is a subset of a simplex isomorphic to $T_k^{n,l}$, $0 \leq k \leq l+1 \leq n$, via some affine isomorphism. It is critical when $l = k-1$ and $k$ is then said to be its index while 
it is regular otherwise. 
\end{definition}

For every $n\geq0$ and  every $k\in \{0,\ldots,n\}$, we also denote the critical Morse tile $T^{n,k-1}_k$ of index $k$ by $C_k^n$ and observe that $T^{n,n-1}_k= T^{n}_{k+1}$, while $T^{-1}_0 = \emptyset$ by convention. In the case $k=0$, $C_0^n=\Delta_n$
is the standard  $n$-simplex while $C^n_n$ is the standard open $n$-simplex. 

The next lemma computes the contribution of each tile to the Euler characteristic of a tiled simplicial complex. Recall that the Euler characteristic is additive and may be computed with respect to the cellular structure
of the simplicial complex, given by open simplices.

\begin{lemma}\label{Lem_chitile}
For every $0\leq k\leq n,$ $\chi(C_k^n)=(-1)^k$. Likewise, for every $0\leq k\leq l \leq n-1,$ $\chi(T^{n}_k \setminus T^{l}_k)=0$.
\end{lemma}

\begin{proof}
If $k=0$, $T^n_k=C^n_0$ is the standard simplex  $\Delta_n$, so that  $\chi(T^n_k)=1$. If $k=n+1$, $T^n_k=C^n_n$ is the open simplex, so that  $\chi(T^n_k)=(-1)^n$. 

For every $n\geq 1,$  $T_1^n=T_0^n\setminus T_0^{n-1}$  has vanishing Euler characteristic. By Proposition \ref{Prop_4.1}, for every $n\geq 2$, $T^n_2=T^n_0\setminus (T_0^{n-1}\sqcup T_1^{n-1})$
 satisfies $\chi(T^n_2)=\chi(T^n_0)-\chi(T^{n-1}_0)-\chi(T^{n-1}_1)=0.$ Then, by induction on $k$, for every $k\geq 1$ and every $n\geq k,$ $T^n_k=T_0^n\setminus (T_0^{n-1}\sqcup\ldots\sqcup T_{k-1}^{n-1})$ has vanishing Euler characteristic so that any basic tile has vanishing Euler characteristic unless it is isomorphic to an open  or a closed simplex. Then, for every $0<k<n$, $\chi(C^{n}_k)=\chi(T^{n}_k)-\chi(T^{k-1}_k)=0-(-1)^{k-1}=(-1)^k$ by definition and  the additivity of the Euler characteristic. Likewise for every $0\leq k\leq l\leq n-1$,
 $\chi(T^{n,l}_n)=\chi(T^{n}_k)-\chi(T^{l}_k)=0.$
\end{proof}

\begin{proposition}\label{Prop_JSkeletonM}
For every $0\leq k\leq j\leq n$, the $j$-skeleton of $C^n_k$ admits a partition by basic tiles isomorphic to $T_l^j$ with $l>k$ and a unique critical Morse tile isomorphic to $C^j_k$. This skeleton is empty if $j<k$. Likewise, for every $0\leq k\leq l< n-1$, the $j$-skeleton of $T_k^{n,l}$ is empty if $j<k$, admits a partition by basic tiles isomorphic to $T^j_m$ with $m>k$ if $k\leq j\leq l$ together with a unique tile isomorphic to $T_k^{j,l}$ if $l<j\leq n$.
\end{proposition}

\begin{proof}
By definition, $C^n_k=T^n_k\setminus T_k^{k-1}$ and by Proposition \ref{Prop_UniqueTile}, the $j$-skeleton of $T^n_k$ is empty if $j<k-1$ and admits a partition by basic tiles isomorphic to  $T_l^j$, $l\geq k$, with a unique tile of order $k$. The latter contains the unique $(k-1)$-dimensional tile of order $k$. The $j$-skeleton of $C^n_k$ thus inherits a partition by tiles isomorphic to  $T^j_l$ with $l>k$ and a tile isomorphic to $T^j_k\setminus T^{k-1}_k=C^j_k.$ In particular,  it is also empty if $j=k-1$ by Proposition \ref{Prop_UniqueTile}. Likewise, $T_k^{n,l}=T^n_k\setminus T_k^l$ and by Proposition \ref{Prop_UniqueTile}, the $l$-skeleton of $T^n_k$ admits a partition by basic tiles isomorphic to $T_m^l$ with $m\geq k$, the tile of order $k$ being unique. The  $l$-skeleton of $T^{n,l}_k$  thus inherits a partition by basic tiles of order $m> k$. It then follows from Proposition \ref{Prop_UniqueTile} that the same holds true for the $j$-skeleton of $T^{n,l}_k$ with $j\leq l$, these skeletons being empty  if $j<k.$ Finally, if $j>l$, we deduce from Proposition \ref{Prop_UniqueTile} that the $j$-skeleton of 
$T^{n,l}_k$  admits a partition by tiles isomorphic to $T^j_m$ with $m>k$ and a tile isomorphic to $T^j_k\setminus T^{l}_k=T_k^{j,l}$. Hence the result.
\end{proof}

 \begin{proposition}\label{Prop_ConeMorse}
For every $0<k<n$, $(c\ast C^n_k)\setminus \{c\}=T^{n+1,k}_k$. Moreover,  this cone deprived of its base $C^n_k$ is $C^{n+1}_{k+1}$. Similarly, for every $k\leq l \leq n-1,$ $(c\ast T^{n,l}_k)\setminus \{c\}= T^{n+1, l+1}_{k}$ and if this cone is deprived of its base $T^{n,l}_k,$ it is $T^{n+1, l+1}_{k+1}$.
\end{proposition}

\begin{proof}
By definition, $C_k^n=T^n_k\setminus T^{k-1}_k$, so that $(c\ast C^n_k\setminus \{c\})=(c\ast T^n_k)\setminus (c\ast T^{k-1}_k)=T^{n+1}_k\setminus T^k_k=T^{n+1,k}_k.$
If the cone is deprived of its base, it gets $T^{n+1}_{k+1}\setminus T^k_{k+1}=C^{n+1}_{k+1}$. Similarly,  $(c\ast T^{n,l}_k)\setminus \{c\}=(c\ast T^n_k)\setminus (c\ast T^l_k)=T^{n+1}_k\setminus T^{l+1}_k=T^{n+1, l+1}_k.$ And if the cone is deprived of its base $T^{n,n}_k$, we get $T^{n+1}_{k+1}\setminus T^{l+1}_{k+1}=T^{n+1, l+1}_{k+1}.$ 
\end{proof}

\subsection{Morse tilings}\label{Ssect_MorseTiling}

We now introduce Morse tilings of finite simplicial complexes, or more generally of Morse tileable sets. For a definition of simplicial complexes, see for instance \cite{M}. The relation between Morse tilings and discrete Morse theory is discussed in section \ref{Ssect_Morse}.

\begin{definition}\label{Defn_Morse tiling}
A subset $S$ of a finite simplicial complex $K$ is said to be Morse tileable iff it admits a partition by Morse tiles such that for every $j\geq 0$, the union of tiles of dimension greater than $j$ is the intersection with $S$ of a simplicial subcomplex of $K$.  It is Morse tiled iff such a partition, called a Morse  tiling, is given. The dimension of $S$ is then the maximal dimension of its tiles. 
\end{definition}

The dimension of a Morse tileable set does not depend on the tiling, for it is also the maximal dimension of the open simplices contained in $S$.
The trivial partition of a simplicial complex by open simplices is thus not a Morse tiling, though open simplices are Morse tiles. Recall that Proposition \ref{Prop_4.1} provides in particular 
a partition of the boundary $\partial \Delta_{n+1}$ which contains each basic tile $T^n_k$ exactly once, $k\in \{0,\ldots, n+1\}$. This is a Morse tiling of a triangulated sphere, even an $h$-tiling by Definition \ref{Defn_$h$tiling}, with one critical tile of index $0$ and one critical tile of index $n$.
Theorem \ref{Thm_intro2} provides in particular Morse tileable triangulations on every closed manifold of dimension less than four.

\begin{definition}\label{Defn_Morsesubset}
Let $S$ be a Morse tiled set. A subset $S'$ of $S$ is said to be a Morse tiled subset iff it is a union of Morse tiles and there exists a subcomplex $L$ of a finite simplicial complex $K$ such that
$S \subset K$ and $S' = S \cap L$.
\end{definition}

By definition thus, if $S$ is a Morse tiled set, then for every $j\geq 0$, the union of its tiles of dimension greater than $j$ is a Morse tiled subset of $S$.

\begin{corollary}\label{Cor_SkeletonM}
Let $S$ be a Morse  tileable set. Then, all its skeletons are Morse tileable. Moreover, given a Morse  tiling on $S$, there exist Morse tilings on its skeletons   $S^{(i)}, i\geq0$, such that every tile of $S^{(i)}$ is contained in a tile of $ S^{(i+1)}$. 
\end{corollary}

\begin{proof}
By definition, $S$ is a subset of a finite simplicial complex $K$.
Let $n$ be the dimension of $S$ and let a Morse tiling be given.  By Proposition \ref{Prop_UniqueTile} and Proposition \ref{Prop_JSkeletonM}, the $(n-1)$-skeleton of every $n$-dimensional tile of $S$ admits a partition by Morse tiles. Then, the union of tiles of $S$ of dimension less than $n$ with the ones given by these partitions induces a partition of $S^{(n-1)}$  with tiles which are either tiles of $S^{(n)}=S$ or contained in such tiles. Moreover, by construction, for every $j\in\{0,\dots, n-1\}$, the union of tiles of dimension greater or equal to  $j$ in this partition is the intersection of $K^{(n-1)}$ with the union of tiles of dimension greater or equal to $j$ in $S$. 
Since the latter is by definition the intersection with $S$ of a subcomplex $L$ of $K$, the former is the intersection with $S$ of the complex of $L^{(n-1)}$, so that $S^{(n-1)}$ is Morse tileable. The result is then obtained by induction, replacing $S$ with $S^{(n-1)}$.
\end{proof}

We prove in subsection \ref{Ssect_TilingThm} that the first barycentric subdivision of a Morse tileable set is also Morse tileable, see Corollary \ref{Cor_BarSubM}, so that 
the class of Morse tileable sets is stable under barycentric subdivisions and skeletons, proving Theorem \ref{Thm_intro0}. We already introduced in \cite{SW3} a notion of tileable simplicial complexes sharing the same properties. Let us recall and slightly generalize this subclass of Morse tileable simplicial complexes.

\begin{definition} \label{Defn_$h$tiling}A subset $S$ of a finite simplicial complex $K$ is said to be $h$-tileable iff it admits a partition by basic tiles such that for every $j\geq 0$, the union of tiles of  dimension greater than $j$ is  the intersection with $S$ of a simplicial subcomplex of $K$. It is $h$-tiled    iff such a partition, called  an $h$-tiling, is given. 
\end{definition}

Definition \ref{Defn_$h$tiling} extends the definition given in \S~4.2  of \cite{SW3} where only simplicial complexes and basic tiles of the same dimension are admitted. Such a tiling is now called an $h$-tiling to avoid confusion with Morse tilings
and due to its close relation with $h$-vectors discussed in \S~4.2 of \cite{SW3}, see subsection \ref{SSect_Packing}.

\begin{corollary}\label{Cor_Skeletonh}
Let $S$ be an $h$-tileable set. Then, all its skeletons are $h$-tileable. Moreover, given an $h$-tiling on $S$, there exist $h$-tilings on its skeletons $S^{(i)}$  such that every tile of $S^{(i)}$ is  contained in  a tile of $S^{(i+1)}$.$\hfill \square$
\end{corollary}

\begin{corollary}\label{Cor_BarSubh}
Let $S$ be an $h$-tileable set, then so is its first  barycentric subdivision $\Sd(S)$.  
$\hfill \square$
\end{corollary}

The proofs of Corollary \ref{Cor_Skeletonh} and Corollary \ref{Cor_BarSubh}   have already been given in \cite{SW3} in the case only  simplicial complexes and tiles of the same dimension are involved. Since they are similar to the ones of Corollary \ref{Cor_SkeletonM} and Corollary \ref{Cor_BarSubM}, we do not repeat them. 

 \subsection{Morse shellability}\label{SSect_Shelling}
 
 We now introduce another subclass of Morse tileable sets, the Morse shellable ones, which plays a role in section \ref{Ssect_Morse} and  contains shellable simplicial complexes.
  
 \begin{definition}\label{Morseshellable}
 A subset $S$ of a finite simplicial complex $K$ is said to be Morse shellable (resp. shellable) iff there exists a Morse tiling on $S$ and 
 a filtration $\emptyset =S_0\subset S_1\subset\ldots\subset S_N=S$ by Morse tiled subsets of $S$ such that for every $i\in \{1,\dots, N\},$ $S_i\setminus S_{i-1}$ is a single Morse tile (resp. basic tile). 
 \end{definition}
 
 A finite simplicial complex $K$ is classically said to be shellable iff there exists an ordering $\sigma_1 , \dots , \sigma_N$ of its maximal simplices such that for every $i \in \{ 2, \dots , N \}$, 
 $\sigma_i \cap (\cup_{j=1}^{i-1} \sigma_j)$ is non-empty and of pure dimension $\dim \sigma_i -1$, see Definition $12.1$ of \cite{Koz} for instance. This means that each simplex $\sigma_1 , \dots , \sigma_N$ is not a proper
 face of a simplex in $K$ and  every simplex in $\sigma_i \cap (\cup_{j=1}^{i-1} \sigma_j)$, $i \in \{ 2, \dots , N \}$, is a face of a  $(\dim \sigma_i -1)$-dimensional one in this intersection. Definition \ref{Morseshellable} extends this classical notion of shellability. Indeed. 
 
\begin{theorem}\label{Lemmashel}
A finite simplicial complex is shellable in the sense of Definition \ref{Morseshellable} iff  there exists an ordering $\sigma_1,\ldots, \sigma_N$ of its maximal simplices  such that for every $i\in\{2,\ldots, N\}$, $\sigma_i\cap(\cup_{j=1}^{i-1}\sigma_j)$ is either empty or of pure dimension $\dim \sigma_i -1$. 
\end{theorem}

 \begin{proof}
 Let $K$ be a finite simplicial complex which is shellable in the sense of Definition \ref{Morseshellable}. There exists then a filtration $\emptyset = K_0  \subset K_1\subset\ldots\subset K_N=K$ of $K$ by finite simplicial complexes together with an $h$-tiling on $K$ such that for every $i\in \{1,\dots, N\},$ $K_i\setminus K_{i-1}$ is a single basic tile $T_i$. Let $\sigma_i$ be the closure of $T_i$ in $K$. Then, $\sigma_i \cap (\cup_{j=1}^{i-1} \sigma_j)$ is either empty or of pure dimension $\dim \sigma_i -1$ by Definition \ref{Defn_Basic}. Moreover, $\sigma_i$ cannot be a proper face of a simplex in $K$ by Definition \ref{Defn_$h$tiling}, since otherwise the union of
 tiles of dimensions greater than $\dim \sigma_i$ would not be a simplicial subcomplex of $K$.
 
 Conversely, let us now assume that there exists an ordering $\sigma_1 , \dots , \sigma_N$ of the maximal simplices of a simplicial complex $K$ such that for every $i \in \{ 2, \dots , N \}$, 
 $\sigma_i \cap (\cup_{j=1}^{i-1} \sigma_j)$ is either empty or of pure dimension $\dim \sigma_i -1$. This means that $\sigma_i \cap (\cup_{j=1}^{i-1} \sigma_j)$ is a union of facets of  $\sigma_i$ so that
 $T_i =  \sigma_i \setminus (\cup_{j=1}^{i-1} \sigma_j)$ is a basic tile. These tiles provide a partition of $K$ and we have to prove that for every $j\geq 0$, the union of tiles of dimension greater than $j$ is a 
 simplicial subcomplex of $K$. We proceed by induction on $i \in \{1, \dots , N \}$. For $i=1$, there is nothing to prove. Assume now that the union $\sqcup_{j=1}^{i-1} T_j$ is an $h$-tiling of $K_{i-1}$. Then, each
 open facet of $\sigma_i $ has to be contained in a tile of dimension at least $\dim \sigma_i $ in $K_{i-1}$, since otherwise it would be the interior of one of the tiles $T_j$, $j<i$, and $\sigma_j$ would be a proper face of $\sigma_i $, a contradiction. Since by induction the union of tiles of dimensions at least  $\dim \sigma_i$ in $K_{i-1}$ is a simplicial subcomplex of $K_{i-1}$, the same holds true then for the union of tiles of dimensions at least  $\dim \sigma_i$ in $K_{i}$, which is a simplicial subcomplex of $K_{i}$. The result follows. 
  \end{proof}

  \begin{remark}\label{Rem_rem0}
The proof of Lemma \ref{Lemmashel} actually gives more, namely that any shelling of a simplicial complex  
$\emptyset=K_0\subset K_1\subset \ldots\subset K_N=K$ in the sense of Definition \ref{Morseshellable} provides a shelling $\sigma_1,\ldots,\sigma_N$ in the classical sense and vice-versa, where for every $i\in\{1,\ldots,N\}$, $\sigma_i$ is the closure of the tile $K_i\setminus K_{i-1}$ and where however the non-emptiness condition in the classical definition of shelling has been removed. The latter allows for the underlying topological space $|K|$ not to be connected, but it also
allows for slightly more shellings, since an open one-dimensional simplex may connect two different shellings as in Figure \ref{Fig_NonCl}. These are the only new kinds of shellings though, since for all other basic tiles, the union of missing facets is connected.
  \end{remark}
  
   \begin{figure}[h]
   \begin{center}
    \includegraphics[scale=0.5]{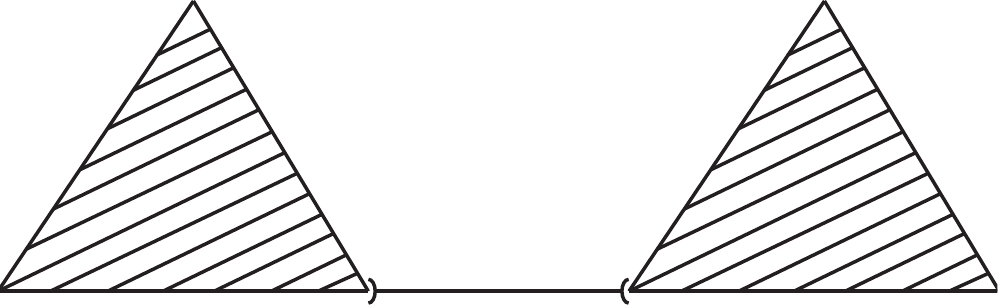}
    \caption{A non classical shelling.}
    \label{Fig_NonCl}
      \end{center}
 \end{figure}

  We may now revisit Proposition \ref{Prop_UniqueTile} and Proposition \ref{Prop_JSkeletonM}.
  
  \begin{lemma}\label{Lem_Co1Skel} 
  The codimension one skeletons of Morse tiles are Morse shellable.
  \end{lemma}
  
  \begin{proof} 
  The $(n-1)$-skeleton of a closed $n$-simplex is well known to be shellable. A shelling is given by Proposition \ref{Prop_4.1},  where the order of the simplices of $\partial \Delta_n=T^{n-1}_0\sqcup\ldots\sqcup T^{n-1}_n$ is given by the order of the tiles in the sense of Definition \ref{Defn_Basic}.  If the Morse tile is basic, isomorphic to  $T^n_k$, then by Proposition \ref{Prop_4.1}, the $(n-1)$-skeleton  of $T^n_k$ is shelled by $T^{n-1}_k\sqcup\ldots\sqcup T^{n-1}_n$, again ordering the tiles by increasing order in the sense of Definition \ref{Defn_Basic}. Finally, if the Morse tile is not basic, isomorphic to $T^{n,l}_k=T^n_k\setminus T^l_k$ with $k-1\leq l<n-1,$ the $(n-1)$-skeleton of $T^{n,l}_k$ is likewise shelled by $T^{n-1,l}_k\sqcup T^{n-1}_{k+1}\sqcup\ldots\sqcup T^{n-1}_n$. 
  \end{proof}
 
 \begin{theorem}\label{Thm_SkelBarSubShell} 
 Let $S$ be a Morse shellable set. Then, all its skeletons are Morse shellable. Moreover, given a Morse shelling on $S$, there exist Morse shellings on its skeletons $S^{(i)}, i\geq0$, such that every tile of $S^{(i)}$ is contained in a tile of $ S^{(i+1)}$. 
 \end{theorem}
 
 \begin{proof}
 Let $S$ be equipped with a Morse shelling. By Definition \ref{Morseshellable}, there exists a filtration $\emptyset\subset S_1\subset\ldots\subset S_N=S$ by Morse tiled subsets of $S$ such that for every $i\in \{1,\dots, N\},$ $S_i\setminus S_{i-1}$ is a single Morse tile.  Let $n$ be the dimension of $S$, it is enough to prove this result for the $(n-1)$-skeleton of $S$, since replacing $S$ by $S^{(n-1)}$ we get the result by decreasing induction. 

We proceed by induction on $i\in \{1,\ldots, N\}$. If $i=1$, $S_1$ is a  single Morse tile and its  $(n-1)$-skeleton is shellable by Lemma \ref{Lem_Co1Skel}. Let us assume now that this result holds true for $S_{i-1}$ and prove it for $S_i$. By the induction hypothesis, the skeleton $S^{(n-1)}_{i-1}$ is shellable and by Lemma \ref{Lem_Co1Skel}, the $(n-1)$-skeleton of the Morse tile $S_i\setminus S_{i-1}$ is  shellable as well.
Then, the concatenation of these shellings provides a shelling of $S_i^{(n-1)}$. Indeed, for every $j\geq 0$, the union of tiles of dimension greater than $j$ in this concatenation is the intersection with $S_i$ of the $(n-1)$-skeleton of $L_j$, where $L_j$ is a subcomplex of a complex $K$ containing $S$ such that the union of tiles of dimension greater than $j$ in $S$ is the trace $L_j \cap S$. 
 \end{proof}
 
 We likewise prove in subsection \ref{Ssect_TilingThm}  that barycentric subdivisions of Morse shellable sets are Morse shellable, see Corollary \ref{Cor_BarSubM}.

\subsection{The tiling theorem}\label{Ssect_TilingThm}

 For every  Morse tile $T=\Delta_n\setminus (\sigma_1\cup\ldots \cup \sigma_k\cup \tau)$, we set $\Sd(T)=\Sd(\Delta_n)\setminus \big(\cup_{i=1}^k \Sd(\sigma_i)\cup\Sd(\tau)\big)$, where $\Sd(\Delta_n)$ denotes the first barycentric subdivision of $\Delta_n$, see \cite{M}.

\begin{theorem}\label{Thm_Sd}
The first barycentric subdivision of every Morse tile $T$ is Morse shellable, shelled by tiles of the same dimension as $T$. Moreover, such a Morse shelling can be chosen such that it contains a critical tile iff $T$ is critical and this critical tile is then unique of the same index as $T$.
\end{theorem}

The fact that the first barycentric subdivision of a basic tile is tileable has already been established in \cite{SW3} and Theorem \ref{Thm_Sd} also recovers the fact that $\Sd(\Delta_n)$ is shellable, see Theorem $5.1$ of \cite{B}.

\begin{proof}
 Let us first prove the result for basic tiles by induction on their dimension $n>0$. If $n=1$, the partitions $\Sd(T_0^1)=T_0^1\sqcup T_1^1$, $\Sd(T_1^1)=2T_1^1$ and $\Sd(T_2^1)=T_1^1\sqcup T_2^1$ are suitable with the filtration $S_1= T^1_0$ and $S_2=\Sd(T^1_0)$,  see Figure \ref{Fig_Sdtile}.

 \begin{figure}[h]
   \begin{center}
    \includegraphics[scale=0.5]{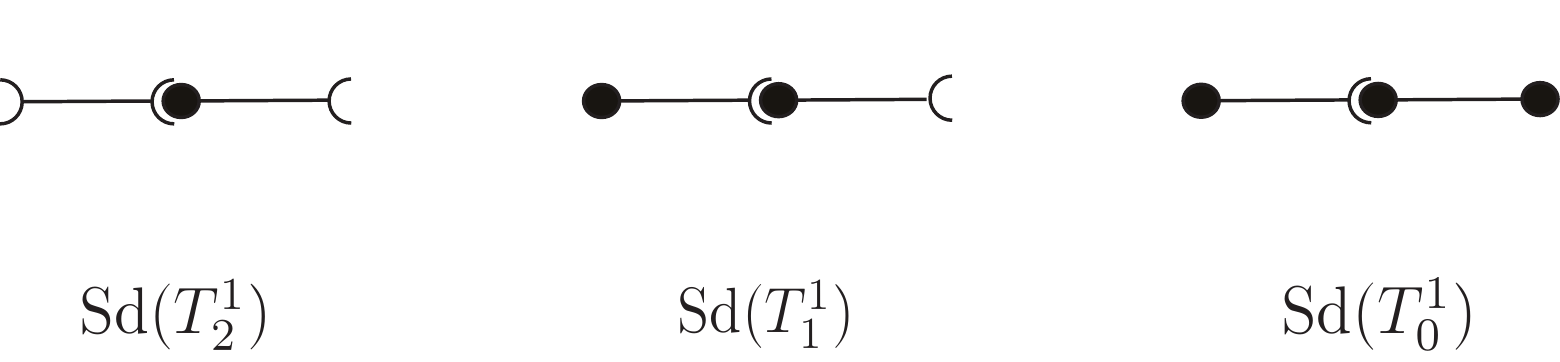}
    \caption{Tilings of subdivided one-dimensional tiles.}
    \label{Fig_Sdtile}
      \end{center}
 \end{figure}

Now, let us assume that the result holds true for $r\leq n-1$ and let us prove it for $r=n$.  From  Proposition \ref{Prop_4.1} (see also
Corollary~4.2 of \cite{SW3}), $\partial \Delta_{n}$ has a partition $\bigsqcup_{k=0}^{n}T_k^{n-1}$ which is such that  for every $r\in \{0,\ldots, n\}$, $\bigsqcup_{k=0}^{r}T^{n-1}_k$ is a subcomplex of $\partial \Delta_{n}$ covered by $r+1$ tiles. We equip $\Sd(\partial \Delta_{n})=\bigsqcup_{k=0}^{n}\Sd(T_k^{n-1})$  with the partition by basic tiles given by the induction hypothesis. There exists a filtration $L_1\subset \ldots \subset L_{(n+1)!}=\Sd(\partial \Delta_{n})$ such that  for every $j\in \{1,\ldots, (n+1)!\}$, $L_j$ is a subcomplex which is the union of  $j$ tiles of the partition. Indeed, if $S^k_i$ is the filtration of $\Sd(T^{n-1}_k)$ given by the induction hypothesis, $k\in \{0,\ldots, n\}$, $i\in \{1,\ldots, n!\}$, we set  for every $j=k n!+i$, $L_j=\bigsqcup_{r=0}^{k-1}T^{n-1}_{r}\sqcup S^k_i$, which is a subcomplex by the induction hypothesis. Then, $\Sd(\Delta_{n})$ gets a partition by cones over the tiles of $\Sd(\partial \Delta_{n})$ centered at the barycenter of $\Delta_{n}$ where all the cones except the one over $T_0^{n-1}$  are deprived of their apex. From Proposition \ref{Prop_4.1} this partition induces a shelling of $\Sd(\Delta_{n})=\Sd(T_0^{n})$ with a unique tile of order zero and no other critical tile,
the cones over the filtration $(L_j)_{j\in \{1,\ldots, (n+1)!\}}$ providing the shelling.
For every $k\in \{1, \ldots, n+1\},$ we equip $\Sd(T_k^{n})=\Sd(\Delta_{n})\setminus \sqcup_{j=0}^{k-1}\Sd(T_j^{n-1})$ with the shelling induced by  removing the bases of all the cones over the tiles included in
 $\bigsqcup_{j=0}^{k-1}\Sd(T_j^{n-1})\subset \Sd(\partial \Delta_{n})$ in the preceding shelling.
From Proposition \ref{Prop_4.1},  these cones deprived of their bases are basic tiles so that we get as well a shelling of $\Sd(T_k^{n})$ which, as in \cite{SW3}, has no more basic tile of order zero as soon as $k>0$
and gets a unique basic tile of order $n+1$ when $k = n+1$. The result is proved in the case of basic tiles.

Let us now prove the result for non-basic tiles  $T^{n,l}_k$. The shelling of $\Sd(T^{n,l}_k)$ is again induced by the one of $\Sd(\Delta_n)$. We obtained the shelling of $\Sd(T^n_k)$  by considering the cones deprived of their bases for every tile of the tiling of $\Sd(T^{n-1}_j)$ with $0\leq j<k$. Among the $(n+1)!$ tiles belonging to the tiling of $\Sd(T^n_k)$, $n!$ are cones deprived of their apex over the tiles of the tiling of $\Sd(T^{n-1}_k)$ and by induction, for every $ k-1\leq l\leq n-1$, $(l+1)!$ of them are iterated cones over the tiles of the tiling of $\Sd(T^l_k)$. If $l=k-1$, the tiling of $\Sd(T^{k-1}_k)$ contains a unique tile $T^{k-1}_k={\stackrel{\circ}\Delta}_{k-1}$ together with tiles $T^{k-1}_m$ with $0<m<k$ by the previous case. Hence, the $k!$ tiles of the tiling of $\Sd(T^n_k)$ which intersect $\Sd(T^{k-1}_k)\subset \Sd(T^n_k)$ consist of a tile $T^n_k$ and tiles $T^n_m$ with $0<m<k$ by  Proposition \ref{Prop_4.1}.
The Morse tiling induced on $\Sd(C^n_k)=\Sd(T^n_k)\setminus \Sd(T^{k-1}_k)$  hence consists of  a tile $C^n_k=T^n_k\setminus T^{k-1}_k$ together with tiles $T_m^{n,k-1}=T^n_m\setminus T_m^{k-1}$ with $0<m<k$ and tiles $T^n_m$ with $0<m\leq n.$  As before, this Morse tiling is a Morse shelling, the Morse shelling being obtained by concatenation.
Finally, if $l\geq k$, the shelling of $\Sd(T^l_k)$ contains tiles $T^l_m$ with $0\leq m\leq l$, the first inequality being strict if $k>0$. The $(l+1)!$ tiles of the tiling of $\Sd(T^n_k)$ which intersect $\Sd(T^l_k)\subset \Sd(T^n_k)$  thus consist of tiles $T^n_m$ with $0\leq m\leq l$ by Proposition \ref{Prop_4.1}, they are iterated cones of the previous ones. The Morse shelling induced on $\Sd(T^{n,l}_k)=\Sd(T^n_k)\setminus \Sd(T^l_k)$ hence consists of tiles $T_m^{n,l}=T^n_m\setminus T^l_m$ with $0<m\leq l$ and of tiles $T^n_j$ with $0<j\leq n$.  \end{proof}

\begin{remark}\begin{enumerate}
\item We actually proved that the regular Morse tiles involved in the partition of $\Sd(C^n_k)$
are either basic, or isomorphic to $T^{n, k-1}_m$ with $0<m<k$. Likewise, the tiles involved in the partition of $\Sd(T^{n, l}_k)$ are either basic or isomorphic to $T^{n,l}_m$ with $0< m\leq l < n-1.$

\item One may check that  $\Sd(C_2^3)$ does not admit any partition involving only critical Morse tiles and basic tiles so that non-basic regular Morse tiles are needed to get Theorem \ref{Thm_Sd}.
\end{enumerate}
\end{remark}

\begin{corollary}\label{Cor_BarSubM}
Let $S$ be a Morse  tileable (resp. shellable) set, then so is its first barycentric subdivision $\Sd(S)$. Moreover, given a Morse  tiling (resp. shelling) on $S$, any induced tiling (resp. shelling) on $\Sd(S)$ contains the same number of critical tiles with the same indices. 
\end{corollary}

\begin{proof} 
Let us first assume that $S$ is a Morse tileable subset of a finite simplicial complex $K$. In order to equip $\Sd(S)$ with a Morse  tiling, we first equip $S$ with a Morse  tiling and then, for each of its tile $T$, equip $\Sd(T)$ with a Morse  tiling given by Theorem \ref{Thm_Sd}. It is indeed a tiling since for every $j \geq 0$, the union of tiles of dimension greater than $j$ of $\Sd(S)$ is the first barycentric subdivision of the union of tiles of dimension greater than $j$ of $S$, so that if the latter is the intersection with $S$ of a subcomplex $L_j$ of $K$, then the former is the intersection with $\Sd(S)$ of the subcomplex $\Sd(L_j)$ of $\Sd(K)$.
 
Let us now prove that the barycentric subdivision of $S$ is shellable, provided $S$ is.  Let then $S$ be equipped with a Morse shelling. By Definition \ref{Morseshellable}, there exists a filtration $\emptyset\subset S_1\subset\ldots\subset S_N=S$ by Morse tiled subsets of $S$ such that for every $i\in \{1,\dots, N\},$ $S_i\setminus S_{i-1}$ is a single Morse tile.
We proceed by induction on $i\in\{1,\ldots, N\}$. If $i=1$, $S_1$ is a closed simplex and the result follows from Theorem \ref{Thm_Sd}.  Let now the result be proved up to the rank $i-1$. Then $S_i\setminus S_{i-1}$ is a Morse tile and we get a shelling of $\Sd(S_i)$ by concatenation of the shelling of $\Sd(S_{i-1})$ with the shelling of $\Sd(S_i\setminus S_{i-1})$ given by Theorem \ref{Thm_Sd}, as in the proof of Theorem \ref{Thm_SkelBarSubShell}. By  Theorem \ref{Thm_Sd}, these induced tiling (resp. shelling) on $\Sd(S)$ contain the same number of critical tiles of the same indices as the one of $S$. Hence the result.
\end{proof}

\subsection{Packings and $h$-vectors}\label{SSect_Packing}

When an $h$-tiling $\mathcal{T}$ of a finite $h$-tileable simplicial complex $K$ only involves tiles of the same dimension $n$, we may encode the number of tiles of each order into the $h$-vector
$h(\mathcal{T}) = (h_0(\mathcal{T}) ,\dots ,  h_{n+1}(\mathcal{T}))$ of the tiling, see Definition $4.8$ of \cite{SW3}. Then, by Theorem $4.9$ of \cite{SW3}, two $h$-tilings $\mathcal{T}$ and $\mathcal{T}'$ of $K$ have the same $h$-vectors provided 
$h_0(\mathcal{T}) = h_0(\mathcal{T}')$ and if moreover $h_0(\mathcal{T}) = 1$, this $h$-vector $h(\mathcal{T})$ coincides with the $h$-vector of $K$ by Corollary $4.10$ of \cite{SW3}, see also \cite{Ful, Z} for a definition. In particular, $h$-tilings provide in this situation a geometric interpretation of the $h$-vector as the number of tiles of each order needed to tile the complex. A part of these results remains valid in the case of Morse tilings. Namely, for every Morse tiling $\mathcal{T}$ on a Morse tileable set, let us denote by $h^j_0(\mathcal{T})$ (resp. $h^j_1(\mathcal{T})$) the number of basic tiles of dimension $j$ and order zero (resp. order one) contained in $\mathcal{T}$, $j \geq 0$.

\begin{proposition}\label{Propf0}
Let $\mathcal{T}$ be a  Morse tiling on an $n$-dimensional  Morse tileable set $S$. Then, $\sum_{j=0}^n (j+1) h^j_0(\mathcal{T}) + h_1 (\mathcal{T}) = f_0 (S)$, where $h_1 (\mathcal{T}) = \sum_{j=0}^n h^j_1(\mathcal{T})$ and $f_0 (S)$ denotes the number of vertices of $S$.
\end{proposition}

\begin{proof}
By Proposition \ref{Prop_UniqueTile} and Proposition \ref{Prop_JSkeletonM}, the only Morse tiles which contain vertices are basic tiles of order zero and one. The former contain $j+1$ vertices if they are of dimension $j$ while the latter contain a single vertex, whatever their dimension is. Counting the number of vertices of $S$ by using the partition $\mathcal{T}$, we deduce the result. 
 \end{proof}
 
\begin{corollary}
Let $\mathcal{T}$ and $\mathcal{T}'$ be two Morse tilings on a Morse tileable set which contain only tiles of the same dimension. Then, $h_0(\mathcal{T}) = h_0(\mathcal{T}')$ if and only if $h_1(\mathcal{T}) = h_1(\mathcal{T}')$. $\square$
\end{corollary}

As in \S $5$ of \cite{SW3}, Morse tilings can be used to produce packings by disjoint simplices in Morse tileable sets.
 
\begin{proposition}\label{Proppack}
Let $\mathcal{T}$ be a Morse tiling on a Morse tileable set $S$. Then, it is possible to pack in $Sd (S)$ a disjoint union of simplices containing, for every $j \geq 0$, at least $ h^j_0(\mathcal{T}) + h^j_1 (\mathcal{T})$ $j$-dimensional ones. 
\end{proposition}

\begin{proof}
A basic tile $T \subset \Delta_j$ of order zero or one contains at least one vertex $v$ and a $j$-simplex of $\Sd(\Delta_j)$ containing $v$ is contained in $\Sd(T)$. A choice of such a simplex for each subdivided basic tile of order zero or one provides a suitable packing. 
 \end{proof}
 
The packings given by Proposition \ref{Proppack} have non-trivial asymptotic under a large number of barycentric subdivisions. Indeed, assume for instance that the Morse tiling $\mathcal{T}$ only contains tiles of the same dimension $n$. Then, by Theorem \ref{Thm_Sd}, $h_0 (\Sd^d (\mathcal{T}))$ is constant, so that by Proposition \ref{Propf0}, $h_0 (\Sd^d (\mathcal{T})) + h_1 (\Sd^d( \mathcal{T})) \sim_{d \to + \infty} f_0 (\Sd^d (S))$, while by
\cite{BW} (see also \cite{DPS, SW1}), $\frac{f_0 (\Sd^d (S))}{f_n (S) (n+1)!^d}$ converges to a positive limit $q_0$ as $d$ grows to $+ \infty$, where $f_n (S)$ denotes the number of $n$-dimensional tiles of $S$. Proposition \ref{Proppack} makes it possible to pack at least a number of disjoint $n$-simplices in $\Sd^d (S)$ asymptotic to $q_0 f_n (S) (n+1)!^{d-1}$ as $d$ grows to $+ \infty$. Such packings where used in 
\cite{SW3} to improve upper estimates on the expected Betti numbers of random subcomplexes in a simplicial complex $K$. More general packing results are obtained in \S~$5$ of \cite{SW3}, where simplices are allowed to intersect each other in  low dimensions. 

\section{Morse shellable triangulations and discrete Morse theory}\label{Ssect_Morse}

\subsection{Discrete Morse theory}\label{subsecdiscrete}

Let us recall few notions of the discrete Morse theory introduced by Robin Forman, see \cite{F1, F2}.
 Let $K$ be a finite simplicial complex. For every $p \geq 0$, we denote  by $K^{[p]}$ its set of $p$-simplices and  for every $\tau,\sigma$ in ${K}$, $\tau>\sigma$ means that $\sigma$ is a face of $\tau$. 

 \begin{definition}[Page 91 of \cite{F1}]\label{def1}
 A function $f:{K}\to \R$ is a discrete Morse function iff for every $p$-simplex $\sigma$ of $K^{[p]}$, the following two conditions are satisfied:
 \begin{enumerate}
 \item $\#\{\tau \in   K^{[p+1]} \,|\,  \tau>\sigma \text{ and } f(\tau)\leq f(\sigma)\}\leq 1$,
 \item $\#\{\nu  \in   K^{[p-1]} \,|\,  \nu<\sigma \text{ and }  f(\nu)\geq f(\sigma\} \leq 1$.
 \end{enumerate}
 \end{definition}
 
  \begin{definition} [Definition~9.1 of \cite{F1}]\label{Defn_9.1}
 
 A discrete vector field on a simplicial complex $K$ is a map $W: K\to K\cup \{0\}$ such that:
 \begin{enumerate}
 \item $\forall p\geq 0$, $W(K^{[p]})\subset K^{[p+1]}\cup \{0\}.$
 \item  For every $\sigma \in K^{[p]}$, either $W(\sigma)=0$ or $\sigma$ is a face of $W(\sigma).$ 
 \item If $\sigma \in \textup{Im}(W)$, then $W(\sigma)=0$.
 \item  For every $ \sigma\in K^{[p]}$, $ \#\{v\in K^{[p-1]}\,|\, W(v)=\sigma\}\leq 1$.
 \end{enumerate}
 \end{definition}
  
 \begin{definition}[Remark on page 131 of \cite{F1}]
 A critical point of a discrete vector field $W$ on a simplicial complex $K$ is a simplex $\sigma\in K$ such that $W(\sigma)=0$ and $\sigma\notin \textup{Im}(W).$
 \end{definition}
 
We set the index of a critical point $\sigma$ of a discrete vector field $W$ to be the dimension of $\sigma$.
 
 \begin{definition}[Definition~6.1 of \cite{F1}]\label{Defn_6.1}
 The gradient vector field of a discrete Morse function $f : K \to \R$ is the discrete vector field $W_f: K\to K\cup \{0\}$ such that for every $p$-simplex $\sigma \in K$,
 $W_f (\sigma) = 0$ if there is no $(p+1)$-simplex $\tau$ such that $\tau >\sigma \text{ and } f(\tau)\leq f(\sigma)$ while $W_f (\sigma) =  \tau$
 otherwise. 
 \end{definition}
 
  \begin{remark}
 The gradient vector field is actually defined on oriented simplices in \cite{F1} and Definition \ref{Defn_6.1} should rather read $W_f (\sigma) = - \langle \partial \tau , \sigma \rangle \tau$ in case $\tau >\sigma \text{ and } f(\tau)\leq f(\sigma)$.
 However, orientations do not play any role throughout this paper. 
  \end{remark}

\begin{definition}[Definition~9.2 of \cite{F1}]\label{Defn_9.2}
Let $W$ be a discrete vector field. A $W$-path of dimension $p$ is a sequence of $p$-simplices $\gamma=\sigma_0, \sigma_1,\ldots, \sigma_r$ such that: 

\begin{enumerate}
\item If $W(\sigma_i)=0$, then $\sigma_{i+1}=\sigma_i$.
\item If $W(\sigma_i)\neq 0$, then $\sigma_{i+1}\neq \sigma_i$ and $\sigma_{i+1}<W(\sigma_i)$ (i.e. $\sigma_{i+1}$ is a facet of $W(\sigma_i)$).
\end{enumerate}

The path $\gamma$ is said to be closed iff $\sigma_r=\sigma_0$ and to be non-stationary iff $\sigma_1\neq \sigma_0$.
\end{definition}

 \begin{remark}\label{RemCW}
These Definitions \ref{Defn_9.1} - \ref{Defn_9.2} are given in \cite{F1} in the more general setting of regular CW-complexes rather than simplicial complexes. They extend to Morse tiled sets in the sense of Definition \ref{Defn_Morse tiling} as well, replacing simplices by their relative interiors. 
 \end{remark}
 
\begin{theorem}[Theorem~9.3 of \cite{F1}]\label{Thm_9.3}
Let $W$ be a discrete vector field on a finite simplicial complex. There is a discrete Morse function $f$ for which $W$ is the gradient vector field if and only if $W$ has no non-stationary closed paths. 
Moreover,  for every such $W$, $f$ can be chosen to have the property that if a $p$-simplex is critical, then $f(\sigma)=p$. 
\end{theorem}

A Morse function given by Theorem \ref{Thm_9.3} is said to be self-indexing. We finally recall that the critical points of a discrete Morse function on a finite simplicial complex span a chain complex which computes its homology, see Theorem $7.3$ of \cite{F1}.

 \subsection{Compatible discrete vector fields}\label{SSect_VF}
 
 We are now going to prove that a Morse tiled set carries discrete vector fields compatible with the tiling, since every Morse tile carries discrete vector fields, see Remark \ref{RemCW}.
 
 \begin{proposition}\label{Prop_DiscreteMorse}
 For every $n\geq 0,$ every $k\in\{0,\ldots, n+1\}$
 and every decomposition $T^n_k=T^{n-1}_k\sqcup\ldots\sqcup T^{n-1}_n\sqcup T^n_{n+1}$ given by Proposition \ref{Prop_4.1}, the tile $T^n_k$ has a discrete vector field $W^n_k$ such that $W^n_k(T^{n-1}_n)=T^n_{n+1}$ and such that for every $l\in \{k\ldots, n-1\}$ the restriction of $W^n_k$ to $T^{n-1}_l$ coincides with one vector field  $W^{n-1}_l$. Such a vector field has no critical point if $0<k<n+1$, a unique critical point of index zero if $k=0$ and a unique critical point of index $n$ if $k=n+1$. \end{proposition}
 
 \begin{proof}
We proceed by induction on $n$. If $n=0$, we set $W^n_k=0$ for every $k\in\{0,1\}$ and the result holds true.
Let us suppose that the result is proved up to the dimension $n-1$ and prove it for the dimension $n$. Let then $k\in \{0,\ldots, n+1\}$ and a decomposition $T^n_k=T^{n-1}_k\sqcup\ldots\sqcup T^{n-1} _n\sqcup T^n_{n+1}$ be chosen (given by Proposition \ref{Prop_4.1}). If $k=n+1,$ we set $W^n_{n+1}=0$  and the tile $T^{n}_{n+1}$ is critical of index $n$ since it has no facet. Otherwise, we set $W^n_k(T_n^{n-1})=T^n_{n+1}$ and for every $l\in\{k,\ldots, n-1\}$ we set the restriction of $W^{n}_k$ to the tile isomorphic to  $T^{n-1}_l$ to be $W^{n-1}_l$ through such an isomorphism.
By the induction hypothesis, it has no critical point, unless $k=0$ where it has a unique critical point of index zero.
 \end{proof}
 
 Proposition \ref{Prop_DiscreteMorse} defines many discrete vector fields on the tile $T^n_k$, $n\geq 0$,  $k\in\{0,\ldots, n+1\}$, which have all been denoted by $W^n_k$. Indeed, such a vector field depends
 on the choice of a partition $T^n_k=T^{n-1}_k\sqcup\ldots\sqcup T^{n-1}_n\sqcup T^n_{n+1}$, but also on a similar choice of a partition of the $(n-1)$-dimensional tiles $T^{n-1}_k , \ldots , T^{n-1}_{n-1}$
 and by induction, on such a choice of an $h$-tiling on all skeletons of $T^n_k$, compare subsection \ref{Ssect_htiling}. In particular, for every face $\tau$ of $\Delta_n$ not contained in $\sigma_1 \cup \dots \cup \sigma_k$, 
 where $T^n_k = \Delta_n \setminus (\sigma_1 \cup \dots \cup \sigma_k)$ and $\dim \tau = l \in \{k , \dots , n-2 \}$, we may choose these partitions in such a way that $\tau \setminus (\sigma_1 \cup \dots \cup \sigma_k)$
 is a basic tile of order $k$ of the $l$-skeleton of $T^n_k$, which is thus preserved by $W^n_k$. Such a vector field $W^n_k$ then restricts to a discrete vector field on the complement 
 $T^{n,l}_k = T^{n}_k \setminus T^{l}_k$.
 
  \begin{corollary}\label{Cor_DiscreteMorse}
 For every $n\geq 0$ and every $k \in \{0,\ldots, n\},$ the critical Morse tile $C^n_k$ inherits from any vector field given by Proposition \ref{Prop_DiscreteMorse} a discrete vector field which has a unique critical point of index $k$. Moreover, for every $0\leq k < l + 1 \leq n$, the standard regular Morse tile $T_k^{n,l}$ inherits from any vector field  given by Proposition \ref{Prop_DiscreteMorse} which preserves $T^{l}_k \subset T^n_k$ a discrete vector field without any critical point.
 \end{corollary}
 
 \begin{proof}
By Proposition \ref{Prop_UniqueTile}, the $k$-skeleton of $T^n_k$ is tiled by a unique tile $T^k_k=T^{k-1}_k\sqcup T^k_{k+1}$ and by Proposition \ref{Prop_DiscreteMorse}$, W^n_k(T_k^{k-1})=T^k_{k+1},$ for any vector filed $W^n_k$ given by this proposition. Thus, $W^n_k$ induces a discrete vector field  on $C^n_k=T^n_k\setminus T^{k-1}_k,$ just by restriction. The tile $T^k_{k+1}\subset C^n_k$ is then critical since it  is no more in the image of $W^n_k,$ so that this vector field on $C^n_k$ has a unique critical point of index $k$. Likewise, the vector field $W^n_k$ of $T^n_k$ preserves $T^{l}_k$ and thus restricts to a vector field on $T_k^{n,l}=T^n_k\setminus T_k^{l}$. By Proposition \ref{Prop_DiscreteMorse}, it has no critical point.
 \end{proof}

\begin{definition}
Let $S$ be a Morse tiled set. A discrete vector field on $S$ is said to be compatible with the tiling iff it preserves the tiles and its restriction to each tile is given by Proposition \ref{Prop_DiscreteMorse} or Corollary \ref{Cor_DiscreteMorse} via some affine isomorphism.
\end{definition}

We deduce the following.
 
 \begin{theorem}\label{Thm_DiscreteMorse}
 Let $K$ be a finite simplicial complex equipped with a Morse  tiling. Then, the critical points of any discrete vector field compatible with the tiling are in one-to-one correspondence with the critical tiles, preserving the index. If moreover such a vector field has no non-stationary closed paths, then it is the gradient vector field of a self-indexing discrete Morse function on $K$ whose critical points are in one-to-one correspondence, preserving the index, with the critical tiles of the tiling.
  \end{theorem}
 
 \begin{proof}
By Definition \ref{Defn_Morse tiling}, the Morse  tiling on $K$ provides a partition of $K$ by Morse tiles. The vector fields given by Proposition \ref{Prop_DiscreteMorse} and Corollary \ref{Cor_DiscreteMorse} thus induce  discrete vector fields on $K$ whose critical points are in one-to-one correspondence with the critical Morse tiles, preserving the index. Now, Theorem \ref{Thm_9.3}
guarantees that such a vector field is the gradient vector field of some discrete self-indexing Morse function on $K$ provided that it has no  non-stationary path. \end{proof}

We finally provide a criterium which ensures that a compatible discrete vector field has no non-stationary closed path. This criterium given by Theorem \ref{Thm_DescreteMorseFiltration} applies to Morse shellings, see Definition \ref{Morseshellable} and Corollary \ref{Cor_Mshellable}.

\begin{lemma}\label{Lemmaclosed}
Every discrete vector field given by Proposition \ref{Prop_DiscreteMorse} or Corollary \ref{Cor_DiscreteMorse} has no non-stationary closed path in the corresponding Morse tile. 
\end{lemma}

 \begin{proof}
 It is enough to prove the result for a basic tile $T^n_k$ equipped with  a discrete vector field $W^n_k$ given by Proposition \ref{Prop_DiscreteMorse}, since vector fields given by Corollary \ref{Cor_DiscreteMorse} on non-basic Morse tiles are restriction of the formers, so that every path on a non-basic Morse tile is also a path on the corresponding basic tile, with the exception of the stationary path at the critical point  in the case of a critical Morse tile. We then prove the result by induction on the dimension $n$ of the tile. If $n=0$, there is nothing to prove, every path is stationary. Otherwise, let us choose a partition $T^n_k=T^{n-1}_k\sqcup\ldots\sqcup T^{n-1}_n\sqcup T^n_{n+1}$  given by Proposition \ref{Prop_4.1}  and an associated discrete vector field $W^n_k$. A path of dimension $n$ of $W^n_k$ is stationary, since $W^n_k (T^n_{n+1})$ has to vanish. 
 A path of dimension $n-1$ which begins with $T^{n-1}_n$ continues in one of the tiles $T^{n-1}_k, \ldots , T^{n-1}_{n-1}$ and is then stationary as in the previous case. Any other path is contained in one of the tiles $T^{n-1}_k,\ldots,T^{n-1}_{n-1}$, so that the result follows from the induction hypothesis. 
 \end{proof}

 \begin{theorem}\label{Thm_DescreteMorseFiltration}
 Let $K_0\subset K_1\subset\ldots\subset K_N=K$ be a filtration of Morse tiled finite simplicial complexes such that for every $i\in \{1,\ldots, N\},$ $K_i\setminus K_{i-1}$ is a single Morse tile. Let $W$ be a compatible discrete vector field on $K$ such that its restriction to $K_0$ has no non-stationary closed path. Then, $W$ has no non-stationary closed path and it is the gradient vector field  of a discrete self-indexing Morse function on $K$.
 \end{theorem}
 
 \begin{proof}
 We prove the result by induction on $i\in \{0,\ldots, N\}.$ If $i=0$, the result holds true by hypothesis. Let $i>0$ and $W$ be a discrete vector field on $K$ compatible with the Morse tiling and whose restriction to $K_0$ has no non-stationary closed path. Then, $K_i\setminus K_{i-1}$ is reduced to a single Morse tile and by Lemma \ref{Lemmaclosed}, it has no non-stationary closed path. Now, a $W$-path on $K_i$ is either contained in $K_i\setminus K_{i-1}$, or it meets $K_{i-1}$ and cannot leave $K_{i-1}$ once it entered in this subcomplex by definition. In both cases, from the induction hypothesis, it cannot have any non-stationary closed path. Hence the result.
 \end{proof}

\begin{corollary}\label{Cor_Mshellable}
Every  discrete vector field compatible with a Morse shelling of a finite simplicial complex is the gradient vector field of a discrete self-indexing Morse function. Moreover its critical points are in one-to-one correspondence  with the critical tiles of the shelling, preserving their indices.
\end{corollary}
 
\begin{proof}
 From Theorem \ref{Thm_DescreteMorseFiltration}, any discrete vector field compatible with any Morse shelling is the gradient vector field of a discrete Morse function, since its restriction to $K_0=\emptyset$ has no non-stationary closed path. The result follows then from Theorem \ref{Thm_DiscreteMorse}.
  \end{proof}
 
\subsection{Proof of Theorem \ref{Thm_intro1}}\label{SSect_intro1}

\begin{proof}
Let $K$ be a finite simplicial complex homeomorphic to a closed surface, which we may assume to be connected. We have to prove that there exists a filtration  $K_1\subset\ldots \subset K_N$ of Morse tiled simplicial complexes such that $K_N=K$ and such that for every $i\in\{1,\ldots, N\},$ $K_i$ is the union of $i$ Morse tiles, see Definition \ref{Morseshellable}. In order to prove the existence of the filtration, we proceed by induction on $i >0$. If $i=1$, we choose any closed simplex in $K$ and declare that $K_1$ is this simplex, tiled by a single critical tile of index 0. Let us assume by induction that we have constructed a tiled subcomplex $K_i$ with $i$ tiles. If there exists an edge $e$ in $K_i$ which is adjacent to only one triangle of $K_i$, we know from the Dehn-Sommerville relations that $K$ contains a triangle $T$ adjacent to $e$ and not contained in $K_i$. Then $T\setminus K_i$ is isomorphic to a triangle deprived of at least one face of dimension one and thus at most one face of codimension  greater than one, so that $T\setminus K_i$ is a Morse tile by Definition \ref{Defn_Morse}. We set $K_{i+1}$ to be the union of $K_i$ and $T$ (together with its faces) and equip it with the Morse tiling given by the one of $K_i$ completed by $T\setminus K_i$. If now all edges of $K_i$ are adjacent to two triangles of $K_i$, let us prove that $K_i=K.$ From the Dehn-Sommerville relations, we know that every edge is adjacent to at most two triangles of $K_i$. We observe that the link of every vertex in $K$ is a triangulated circle, so that the star of a vertex in $K$ is a cone over a polygone, see Figure \ref{Fig_Star}.

\begin{figure}[h]
   \begin{center}
   \includegraphics[scale=0.4]{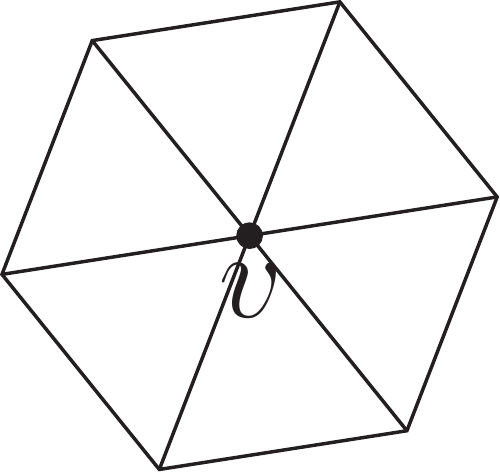}
   \caption{The star of a vertex $v$ in $K$.}
     \label{Fig_Star}
        \end{center}
 \end{figure}
 
 Let $v$ be a vertex in $K$. Since the underlying topological space $|K|$ is connected, there exists a path $v_0,v_1,\ldots, v_k$ such that $v_0\in K_i, v_k=v$ and for every $j\in \{0,\ldots, k-1\},$  $[v_j,v_{j+1}]$ is and edge of $K$.
 Then, by construction, $v_0$ is adjacent to a triangle of $K_i$ and since all edges of $K_i$ are adjacent to two triangles, all triangles adjacent to $v_0$ have to be in $K_i$, see Figure \ref{Fig_Star}. Thus $v_1$ belongs to $K_i$ as well and by induction, $v$ belongs to $K_i$. Hence, $K_i$ contains all vertices of $K$ and also all triangles and edges adjacent to them, so that $K_i=K.$ 
 The proof is similar in dimension 1.   
\end{proof}

The proof of Theorem \ref{Thm_intro1} is algorithmic and the shellings it provides do not use any regular Morse tile of vanishing order.

\subsection{Morse tilings on triangulated handles}\label{SSSect_Mhandles}

 Recall that a handle of index $i$ and dimension $n$ is by definition a product of an $i$-dimensional disk with an $(n-i)$-dimensional one, see \S~6 of \cite{RS}.
 We likewise define a handle of index $i$ in discrete geometry to be the product of simplices ${{\Delta}}_i\times \Delta_{n-i}$, or rather in what follows the product ${\stackrel{\circ}{\Delta}}_i\times \Delta_{n-i}$ of an open simplex of dimension $i$ with a closed $(n-i)$-simplex, suitably triangulated. Our purpose is to define a Morse  shelling on such triangulated $i$-handle for $i=1$ or $n-1$, the general case being postponed to \cite{Welsch}.
 
 \begin{proposition}\label{Prop_Handle}
 For every $n\geq 2$, $\Delta_1\times \Delta_{n-1}$ has a subdivision into $n$ simplices  $\sigma_1,\ldots, \sigma_n$ of dimension $n$ turning it into a shellable simplicial complex. Moreover, writing $\partial \Delta_1=\{0,1\},$ it can be chosen in such a way that for every $i\in \{1,\ldots, n\}$, $\dim(\sigma_i\cap (\{0\}\times\Delta_{n-1}))=n-i$ and $\dim(\sigma_i\cap (\{1\}\times\Delta_{n-1}))=i-1$. For every $i\in\{1,\ldots, n\}$, the subcomplex $K_i^n = \sigma_1\cup\ldots\cup\sigma_i$ inherits the $h$-tiling made of one basic tile of order zero and $i-1$ basic tiles of order  one. 
 \end{proposition}

 \begin{proof} If $n=2,$ the square $\Delta_1\times \Delta_1$ is the union of two triangles meeting along a diagonal and the result follows.
 If $n>2,$ let $c$ be a vertex of $\{0\}\times \Delta_{n-1}$ so that this simplex  is the cone $c\ast(\{0\}\times \Delta_{n-2})$ over its facet $\Delta_{n-2}$. Then, the convex domain $\Delta_1\times \Delta_{n-1}$ is a cone centered at $c$ over the base $(\Delta_1\times \Delta_{n-2})\cup (\{1\}\times \Delta_{n-1})$. By induction, the lateral part $\Delta_1\times \Delta_{n-2}$ has a subdivision $\sigma'_1\cup\ldots\cup \sigma'_{n-1}$ such that for every $i\in \{1,\ldots, n-1\},$ $\dim(\sigma'_i\cap (\{0\}\times \Delta_{n-2}))=n-1-i$
and $\dim(\sigma'_i\cap (\{1\}\times \Delta_{n-2}))=i-1$ and such that $\sigma'_1\cup\ldots\cup\sigma'_i= T_0^{n-1}\sqcup T^{n-1}_1\sqcup\ldots\sqcup T^{n-1}_1$.
We then set, for every $i \in \{1,\ldots, n-1 \},$ $\sigma_i=c\ast\sigma'_i$ and $\sigma_n=c\ast (\{1\}\times \Delta_{n-1})$, see Figure \ref{Fig_Handle}. The result  follows, since $(\{1\}\times \Delta_{n-1})\setminus (\{1\}\times \Delta_{n-2})$ is isomorphic to $T_1^{n-1}$ and the cone over a basic tile of order one remains a basic tile of order one by Proposition \ref{Prop_4.1}.
 \end{proof}

\begin{figure}[h]
   \begin{center}
       \includegraphics[scale=0.35]{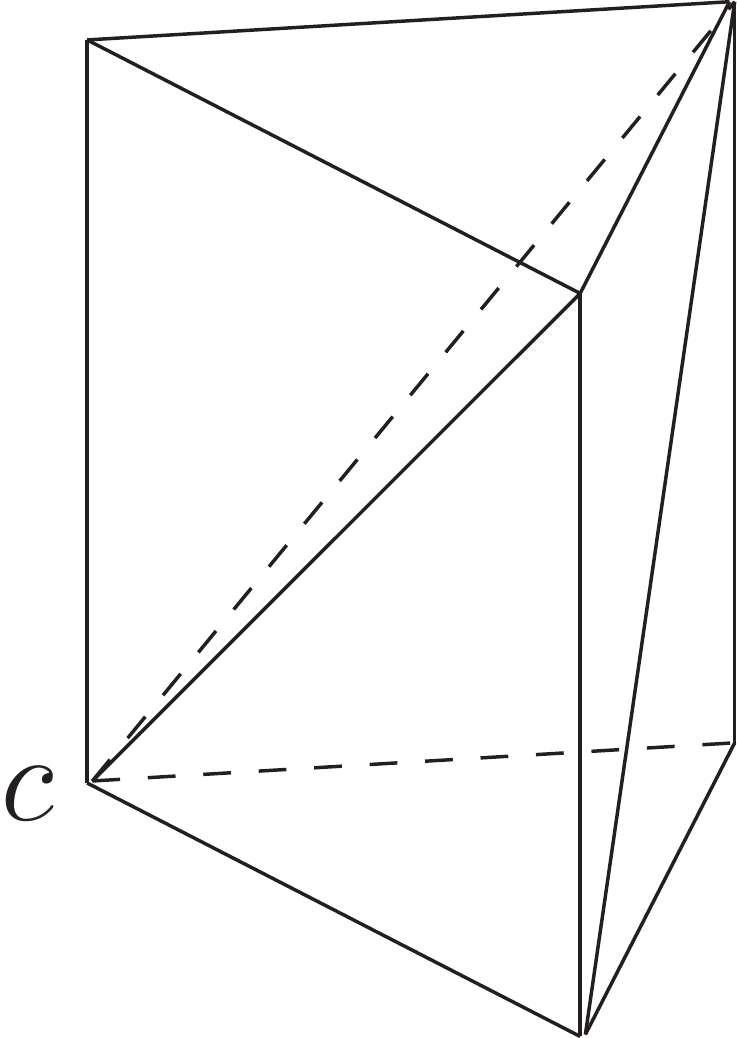}
   \caption{A triangulation on $\Delta_1\times \Delta_2$.}
     \label{Fig_Handle}
        \end{center}
 \end{figure}

\begin{corollary}\label{Cor_Handle}
For every $n\geq 2$, the handles $\stackrel{\circ}{\Delta}_1\times \Delta_{n-1}$, $\Delta_1\times\stackrel{\circ}{\Delta}_{n-1}$ and the product $T^1_1\times\Delta_{n-1}$ inherit from Proposition \ref{Prop_Handle} 
the structure of Morse shellable sets. Moreover, for every $i\in \{1,\ldots, n\}$, the subset $K_i^n \cap (\stackrel{\circ}{\Delta}_1\times \Delta_{n-1})$ gets tiled by a disjoint union $\sqcup_{j=0}^{i-1}T^{n,j}_1$, $K_i^n \cap (\Delta_1\times\stackrel{\circ}{\Delta}_{n-1})$ by one critical tile of index $n-1$ and $i-1$ basic tiles of order $n$ and $K_i^n \cap (T^1_1\times\Delta_{n-1})$ by basic tiles of order one, where $K_i^n$ is the subcomplex given by Proposition \ref{Prop_Handle}. 
\end{corollary}

Corollary \ref{Cor_Handle} thus  provides a Morse shelling on the triangulated one-handle $\stackrel{\circ}{\Delta}_1\times \Delta_{n-1}$ (resp. on the triangulated $(n-1)$-handle $\Delta_1\times\stackrel{\circ}{\Delta}_{n-1}$) containing a unique critical tile, of index one (resp. of index $n-1$).

\begin{proof}[Proof of Corollary \ref{Cor_Handle}] By Proposition \ref{Prop_Handle}, for every $i\in \{1,\ldots, n\}$, $K_i^n =\sigma_1\cup\ldots\cup \sigma_i=T^n_0\sqcup T^n_1\sqcup \ldots\sqcup T^n_1$ and $\dim(\sigma_1\cap (\{0\}\times \Delta_{n-1}))=n-1,$ so that $\{0\}\times \Delta_{n-1}$ is contained in $\sigma_1$ and disjoint from the tiles $T^n_1.$ Thus, $K_i^n \cap (T_1^1\times \Delta_{n-1})=K_i^n \setminus (\{0\}\times \Delta_{n-1})$ inherits the $h$-tiling $\big(\sigma_1\setminus (\{0\}\times \Delta_{n-1})\big) \sqcup_{j=2}^i (\sigma_j\setminus K_{j-1}^n)$ made of $i$ basic tiles of order one. The last part of Corollary \ref{Cor_Handle} is proved.
By Proposition \ref{Prop_Handle} now, $\Delta_1\times \Delta_{n-1}=\sigma_1\cup\ldots\cup\sigma_n$ with $\dim(\sigma_i\cap (\{1\}\times \Delta_{n-1}))=i-1$
 so that by induction on $i\in\{1,\ldots, n\},$ the intersection of $\{1\}\times \Delta_{n-1}$ with  $\sigma_i$ is a face of dimension $i-1$ not contained in  $\sigma_1\cup\ldots\cup\sigma_{i-1}\cup( \{0\}\times\Delta_{n-1})$. 
 By Definition \ref{Defn_Morse}, the 1-handle $\stackrel{\circ}{\Delta}_1\times \Delta_{n-1}$ thus inherits the Morse tiling $\sqcup_{j=0}^{n-1}T^{n,j}_l$ 
made of one critical tile $C^n_1$ of index one and regular Morse tiles $T^{n,l}_1$ with $l\in \{1,\ldots, n-1\}$ and moreover for every $i\in \{1,\ldots, n\}$, $K_i^n \cap (\stackrel{\circ}{\Delta}_1\times \Delta_{n-1}) =
\sqcup_{j=0}^{i-1}(T^{n,j}_1)$.

Let us finally, prove the result for the $(n-1)$-handle $\Delta_1\times \stackrel{\circ}{\Delta}_{n-1}$ by induction on $n.$ For $n=2,$ it has already been proved in the first part. In general, as in the proof of Proposition \ref{Prop_Handle}, let $c$ be a vertex of $\{0\}\times \Delta_{n-1}$ so that $\Delta_1\times \stackrel{\circ}{\Delta}_{n-1}$ is the union of the cone $c\ast (\Delta_1\times \Delta_{n-2})$ over the lateral face deprived of its base and apex and the cone $c\ast (\{1\}\times \stackrel{\circ}{\Delta}_{n-1})$ over the upper face. The latter is isomorphic to a standard tile $T_n^{n}$ by Proposition \ref{Prop_4.1} while by the induction hypothesis, the former is the union of one critical tile $C^n_{n-1}=(c\ast C^{n-1}_{n-2})\setminus C^{n-1}_{n-2}$ and $n-2$ basic tiles $T^n_n=(c\ast T^{n-1}_{n-1})\setminus T^{n-1}_{n-1}$, by Proposition \ref{Prop_ConeMorse} and Proposition \ref{Prop_4.1}. The same induction provides the result since for every $i\in \{1,\ldots, n-1\}$,  $K_i^n = c \ast K_i^{n-1} $.  \end{proof}

\subsection{Simple triangulations on an annulus}\label{SSect_Simple}
In dimension three, the tiled two-handle $\Delta_1\times\stackrel{\circ}{\Delta}_2$ given by Corollary \ref{Cor_Handle} is obtained from the triangulated three-ball $\Delta_1\times \Delta_2$ given by Proposition \ref{Prop_Handle} by removing the cylinder $\Delta_1\times \partial \Delta_2.$ The latter inherits a triangulation with six triangles, see Figure \ref{Fig_Handle}. Each of these triangle has an edge on the boundary component $\{0\}\times \partial \Delta_2$ or $\{1\}\times \partial \Delta_2$ and  the opposite vertex on the opposite component.  We label by $d$ (res. $u$) the  triangles having an edge on the boundary component $\{0\}\times \partial \Delta_2$ (resp. $\{1\}\times \partial \Delta_2$) and choose the standard orientation on $\partial \Delta_2$, so that these six triangles produce the cyclic word $w = ududdu$ while this cyclic word encodes in a unique way the triangulation up to homeomorphisms preserving the boundary components and the orientation, see Figure \ref{Fig_UpDown}.

\begin{figure}[h]
   \begin{center}
       \includegraphics[scale=0.4]{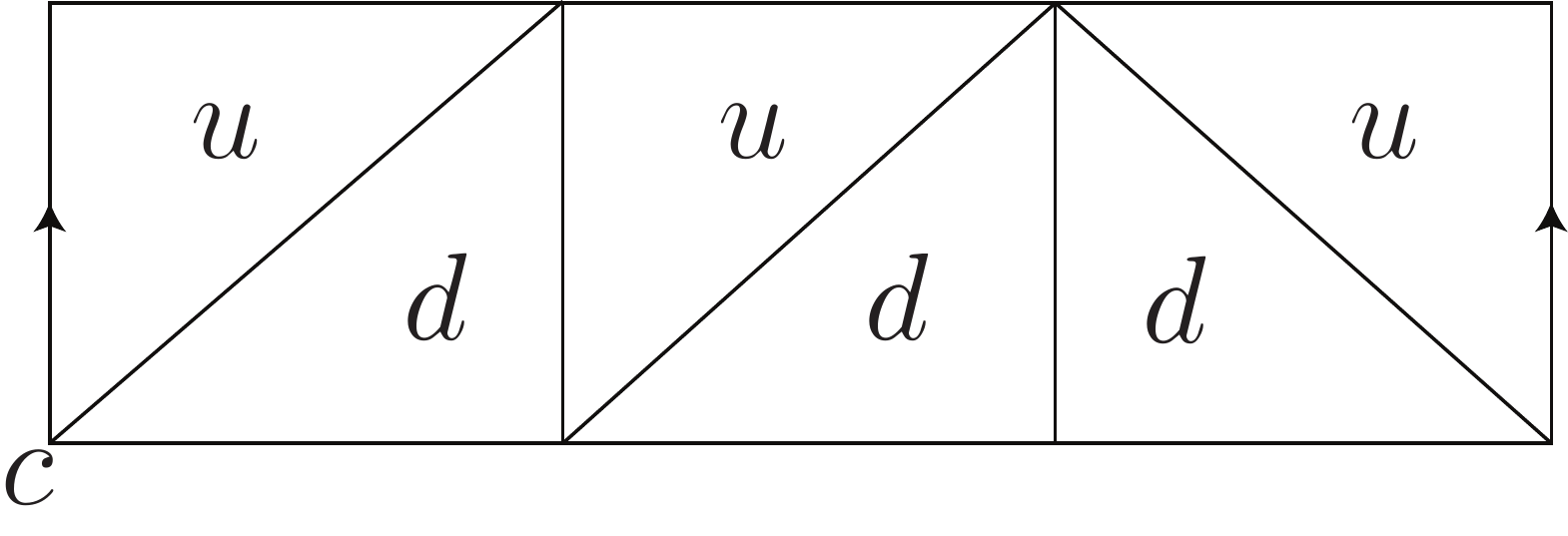}
   \caption{A simple triangulation on $\Delta_1\times \partial \Delta_2$ encoded by $ududdu$.}
     \label{Fig_UpDown}
        \end{center}
 \end{figure}

\begin{definition} \label{Defn_Simple}
A triangulation of an annulus $A \cong [0,1] \times \R/\Z$ is simple iff each triangle has one edge on one boundary component of $A$ and the opposite vertex on the other boundary component. 
\end{definition}

\begin{proposition}\label{Propwords}
Simple triangulations of an annulus up to homeomorphisms preserving the boundary components and the orientation are in one-to-one correspondence with finite cyclic words in the alphabet $\{ d,u \}$ containing each letter at least three times. 
\end{proposition}

\begin{proof}
Let us encode one boundary component of the annulus by the letter $d$ and the other one by the letter $u$. A triangulation of the annulus is a homeomorphism with a two-dimensional simplicial complex and if
this triangulation is simple each triangle of this complex has one edge mapped to some boundary component and thus encoded by either $d$ or $u$ and the opposite vertex on the other component. 
We may join the middle points of the two remaining edges by some arc in the triangle. The union of all these arcs then gives a closed curve homotopic to the boundary components and choosing an orientation
on this curve, we read on it  a finite cyclic word in the alphabet $\{ d,u \}$. Each boundary component has to contain at least three edges so that this cyclic word has to contain each letter at least three times. Conversely, one may reverse the procedure to associate to every such cyclic word a simple triangulation on the annulus, which is uniquely defined
up to homeomorphisms preserving the boundary components and the orientation.
\end{proof}

\begin{lemma}\label{Lem_Words} There exists four cyclic words in the alphabet $\{d,u\}$ which contain each letter three times. Namely,  $w=ududdu, \bar{w}=duduud, uuuddd$ and $ududud.$  

\end{lemma}
\begin{proof}
There is indeed a unique word containing the sequence $uuu$, there are two words containing the sequence $uu$ but not $uuu$, for they have to contain the sequence $duud$ and there is a single word which does not contain the sequence $uu$. These four words are thus $uuuddd, w=duudud, \bar{w}=duuddu$ and $ududud.$  
\end{proof}
Let us now declare that a \emph{compression} of two letters of a cyclic word in the alphabet $\{ d,u \}$ is the replacement of a sequence $dd$ (resp. $uu$) by the single letter $d$ (resp. $u$), while a \emph{suppression}
is the replacement of $udu$ (resp. $dud$) by $ud$ (resp. $du$). The following lemmas will be useful in the proof of Theorem \ref{Thm_intro2}.

\begin{lemma}\label{Lem_w2}
It is possible to reduce any finite cyclic word in the alphabet $\{ d,u \}$ containing each letter at least three times to a word  of six letters given by Lemma \ref{Lem_Words} by applying finitely many compressions or suppressions.
\end{lemma}

\begin{proof}
If the word contains only six letters, there is nothing to prove. Otherwise it contains one letter at least four times, say the letter $u$. Then, either it contains the sequence $uu$ or $dd$ and a compression decreases the number of letters in this word, or it contains the sequence $udu$ and a suppression decreases the number of letters by one. The result follows thus by induction. 
\end{proof}

Let us also declare that a \emph{subdivision} of such a word is the replacement of the letter $u$ by the sequence $duud$ and of the letter $d$ by $dd$.

\begin{lemma}\label{Lem_Word}
It is possible to reduce any finite cyclic word in the alphabet $\{d,u\}$  containing each letter at least three times to $w$ by applying finitely many compressions or suppressions and a single subdivision. 
\end{lemma}
\begin{proof}
By Lemma \ref{Lem_w2}, we may assume that the word has six letter and by Lemma \ref{Lem_Words} 
that it is one of the four words $ududud, uuuddd$, $w$  and  $\bar{w}$. Now, after a single subdivision, these words can be reduced to $w$ using compressions and suppressions. For example, $uuuddd$ becomes $duudduudduuddddddd$. By selecting the first sequence $duu$ and compressing the other ones, it reduces to $duududud$. After two suppressions $dud\to du$ and $udu\to ud$, we get $w$. The proof is similar for the other words.\end{proof}

Let us finally observe that if we perform a barycentric subdivision on an annulus equipped  with a simple triangulation, then this annulus becomes the union of two triangulated annuli separated by the triangulated circle joining the barycenters of the triangles and of the inner edges, see Figure \ref{Fig_Subdiv}.

  \begin{figure}[h]
   \begin{center}
       \includegraphics[scale=0.4]{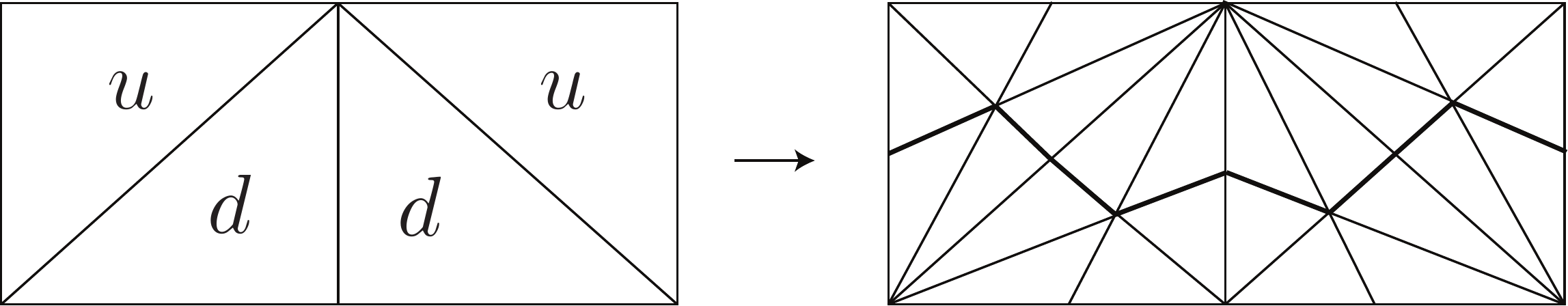}
   \caption{Barycentric subdivision on a simple triangulation.}
     \label{Fig_Subdiv}
        \end{center}
 \end{figure}
  
  One of these annuli contains the subdivided boundary component  labelled by $u$ while  the other  one contains the subdivided boundary component  labelled by $d$. Now, if the original simple triangulation  is encoded by some cyclic word $w'$ in the alphabet $\{d,u\}$, the word encoding the former is deduced  from $w'$ by performing a single subdivision and the latter by  performing the similar substitutions $d\to uddu$ and $u\to uu$. 
  
\subsection{Proofs of Theorem \ref{Thm_intro2} and Corollary \ref{Cor_intro}}\label{SSect_intro2}

We need the following lemma.

\begin{lemma}\label{Lem_KL}
Let $(K,L)$ be a simplicial complex homeomorphic to a compact $n$-dimensional manifold with boundary $(M,\partial M).$ Then, any simplex of $\Sd(K)$ intersects $\Sd(L)$ along a single face, possibly empty.  Moreover, $\Sd(L)$ contains two disjoint $(n-1)$-simplices in each connected component such that  any simplex 
of $\Sd(K)$ intersects  their union along a single face, possibly empty. \end{lemma}

\begin{proof}
By definition, for every simplex $\tau$ of the first barycentric subdivision $\Sd(K)$ of $K$, there exists a collection of simplices  $\sigma_0,\ldots, \sigma_n$  of $K$ such that for every $0\leq i<j\leq n$, $\sigma_i$ is a face of $\sigma_j$ and  such that  the vertices of $\tau$ are the barycenters $\hat{\sigma}_i$  of $\sigma_i$, $i\in\{0,\ldots,n\}$, see~\cite{M}.
If $\tau\cap \Sd(L)\neq \emptyset$, let $j\in \{0,\ldots, n\}$ be the greatest element such that $\hat{\sigma}_j\in\Sd(L)$. Then, $\sigma_j$ is in $L$, so that its faces $\sigma_0,\ldots,\sigma_{j-1}$ are in $L$ as well and  $\hat{\sigma}_0,\ldots, \hat{\sigma}_{j-1}$ are in $\Sd(L)$. The intersection of the simplex $[ \hat{\sigma}_0,\ldots,  \hat{\sigma}_n]$ with $\Sd(L)$ is then the face  $[ \hat{\sigma}_0,\ldots,  \hat{\sigma}_j]$.  

Now, $L$ contains  at least two $(n-1)$-simplices $\sigma, \theta$ in each connected component and $\sigma$ (resp. $\theta$) contains a vertex $\sigma_0$ (resp. $\theta_0$) which is not in $\theta$ (resp. not in $\sigma$). Choosing flags $\sigma_0<\sigma_1<\ldots<\sigma_{n-1}=\sigma$ and $\theta_0<\theta_1<\ldots<\theta_{n-1}=\theta$, we get two disjoint $(n-1)$-simplices $[\hat{\sigma}_0, \ldots,\hat{\sigma}_{n-1}]$ and $[\hat{\theta}_0, \ldots,\hat{\theta}_{n-1}]$  in each connected component of $\Sd(L).$ These simplices have the required property from  what precedes. Hence the result. 
\end{proof}

\begin{proof}[Proof of Theorem \ref{Thm_intro2}] 
Let a smooth Morse function be given on the closed $n$-dimensional manifold $X$, $n\leq 3$. We know from Morse theory that it induces a handle decomposition of $X$, that is, it decomposes $X$  into finitely many sublevels, starting from the empty set and ending with $X$ in such a way that one passes from a sublevel to the next one by attaching some handle, see Theorem~3.2 of \cite{Milnor2} or also \cite{CG, RS}. We are going to prove the result by having each sublevels being triangulated and equipped with a Morse  shelling  and by performing each handle attachment by gluing a Morse  tiled handle. Moreover, we will check that  the  shelling can be extended through the handle to get the result by finite induction.

There is no obstruction to attach a zero-handle whatever $n$ is, it consists in adding a disjoint closed $n$-simplex to the triangulation. Lemma \ref{Lem_KL} makes  it possible to attach a one-handle whatever $n$ is as well, after performing one barycentric subdivision. Indeed, let $X_0$ be an $n$-dimensional simplicial complex homeomorphic to a  manifold with boundary and  let it be equipped with a Morse shelling. By Corollary \ref{Cor_BarSubM}, $\Sd(X_0)$ inherits a Morse shelling with the same number of critical tiles and same indices as  $X_0$, while $|\Sd(X_0)|$ and $|X_0|$ are homeomorphic to each other. By Lemma \ref{Lem_KL}, we may find two disjoint $(n-1)$-simplices in the boundary of $\Sd(X_0)$, each simplex being chosen in any of the connected component of this boundary, such that every simplex of $\Sd(X_0)$ intersects their union along a single face, possibly empty. There is then no obstruction to attach the one-handle $\stackrel{\circ}{\Delta}_1\times \Delta_{n-1}$ along these two $(n-1)$-simplices and to equip this handle with the triangulation and Morse shelling  given by Corollary \ref{Cor_Handle}. Moreover, attaching the $n$-tiles given by  Corollary \ref{Cor_Handle} one after the other to $\Sd(X_0)$, following the shelling order, we extend  the Morse shelling of $\Sd(X_0)$ through the handle  to get a Morse shelled simplicial complex $\Sd(X_0)\sqcup (\Delta_1\times \Delta_{n-1})$
 homeomorphic  to the manifold with  boundary obtained after attaching  the one-handle to $X_0$.  This already proves Theorem \ref{Thm_intro2} in dimension $n=1$.
 
 Let us now prove that it is possible to attach a Morse shelled two-handle in dimension three. Let $X_1$ be a three-dimensional simplicial complex homeomorphic to a manifold with boundary and equipped with a Morse shelling. 
The two-handle has to be attached along a tubular neighborhood  of a two-sided closed curve $C'$ embedded in a boundary component $\Sigma$ of $X$. By Theorem~A1 of~\cite{Ep1} it can be assumed to be the image of a $PL$-embedding of $S^1$, deforming it by some ambient isotopy if necessary. We then perform a large number $d\gg 0$ of barycentric subdivisions on $X_1$ and isotope slightly $C'$ to get a new curve $C$ which does not contain any vertex of  $\Sd^d(\Sigma)$, is transverse to its edges and is such that for every triangle $T$ of $\Sd^d(\Sigma)$, either $C$ is disjoint from $T$, or $C$ intersects $T$ along a connected piecewise linear arc joining two different edges, see Figure \ref{Fig_Curve}.

\begin{figure}[h]
   \begin{center}
    \includegraphics[scale=0.4]{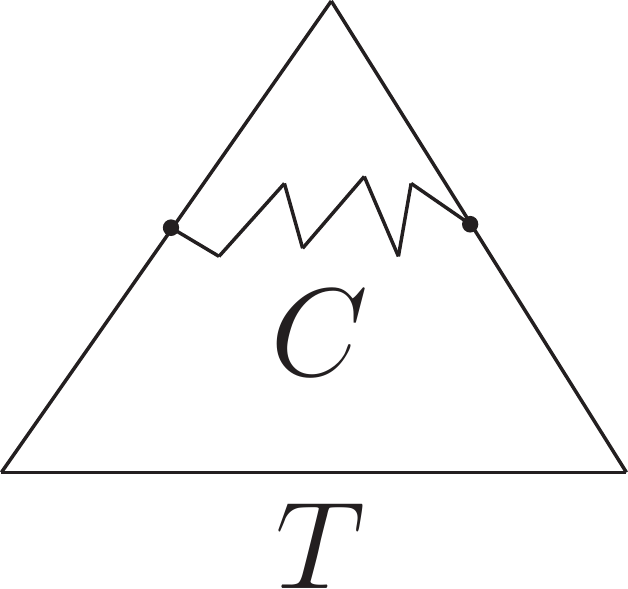}
    \caption{A piecewise-linear arc joining two edges.}
    \label{Fig_Curve}
      \end{center}
 \end{figure}

The union of all  triangles meeting $C$ is then a regular neighborhood $N$ homeomorphic to an annulus  equipped with a simple triangulation, see Definition \ref{Defn_Simple}. Let us choose a homeomorphism between $N$ and the annulus $\Delta_1\times \partial \Delta_2$ lying on the boundary of $\Delta_1\times \Delta_2$. It induces a labelling of the boundary component of $N$ by the letters $d$ and $u$ as well as an orientation  on $C$, for those have been fixed in subsection \ref{SSect_Simple}, while such labelling and orientation encode  the homeomorphism in a unique way up to isotopy. By Proposition \ref{Propwords}, the simple triangulation of $N$ gets then encoded by a cyclic word $w_N$ in the alphabet $\{d,u\}$ containing each letter at least three times.
Moreover, by Lemma \ref{Lem_Word},  this word can be reduced to $w=ududdu$ using finitely many suppressions or compressions together with a single subdivision. We are going to prove that these operations can be performed by attaching basic tiles of order two along $N$ or by performing one barycentric subdivision. To begin with, we may assume 
by Lemma \ref{Lem_KL} that any simplex of  $X_1$ intersects $N$ along a single face, possibly empty, replacing $X_1$ by $\Sd(X_1)$ if necessary. Then, if two consecutive triangles of $N$  are encoded by $dd$ (resp. $uu$), we may attach a basic tile of order two along them to get a new Morse shelled  simplicial complex together with annulus  $N'$ as its boundary which gets encoded by the word obtained  from $w_N$ after the compression $dd\to d$ (resp. $uu\to u$).  This new simplicial complex is homeomorphic to $X_1$  by Lemma~3.25 of \cite{RS}.  As long as the cyclic word contains at least four times the letter $d$ (resp. $u$), there is no obstruction to perform the attachement, since a simplex of $X_1$ can intersect the union of the two consecutive triangles only along a single face. Moreover, we keep the property that any simplex intersects $N'$ along a single face, possibly empty, since $N'$ contains only one facet of the basic tile of order two, see Figure \ref{Fig_DD}.

\begin{figure}[h]
   \begin{center}
    \includegraphics[scale=0.4]{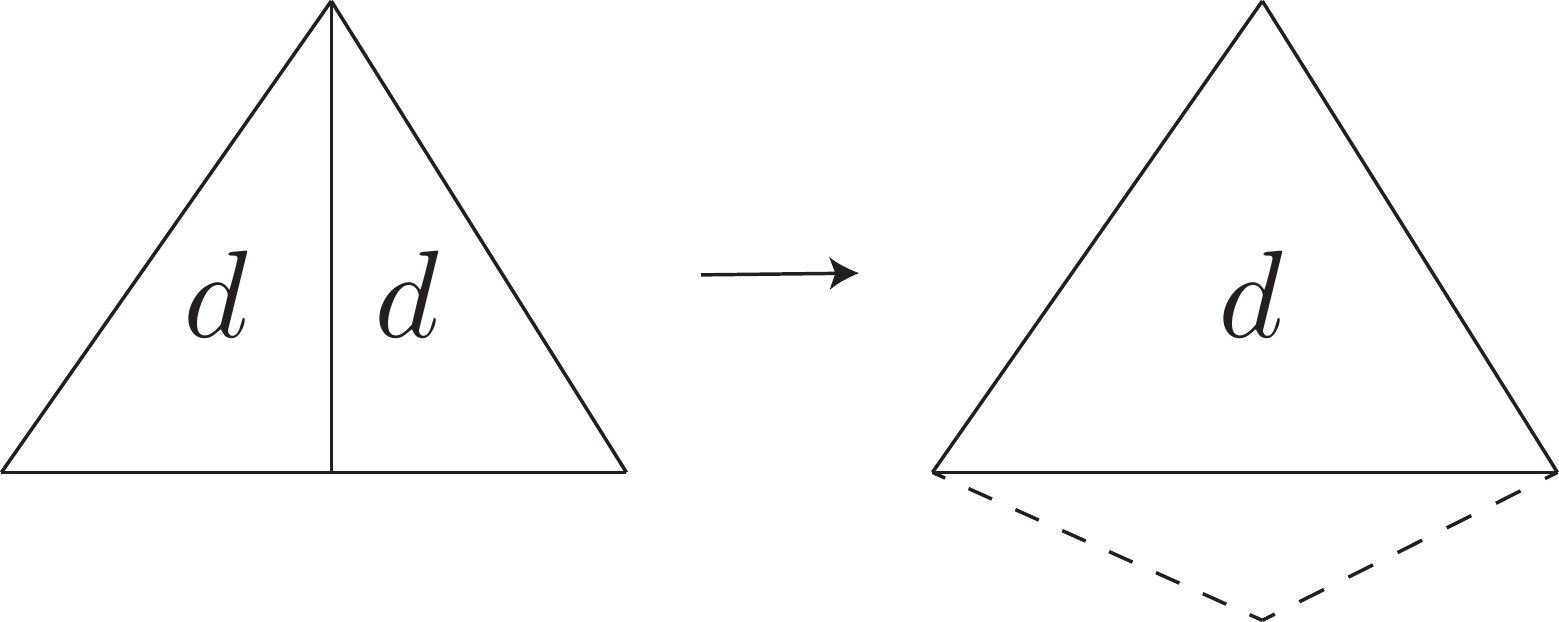}
    \caption{Compression $dd\to d$.}
    \label{Fig_DD}
      \end{center}
 \end{figure}

Likewise, if three consecutive triangles of $N$ are encoded by $udu$ (resp. $dud$), we may attach two basic tiles of order two, first along $du$ (resp. $ud$) to get $uud$ (resp. $ddu$) and then along $uu$ (resp. $dd$) to get $ud$ (resp. $du$), see Figure \ref{Fig_UD_DU}. Again,  there is no obstruction to perform these attachments as long as the cyclic word contains the letter  $u$ (resp. $d$) at least four times and we get a new Morse shelled simplicial complex and neighborhood $N'$ whose triangulation gets encoded by the word obtained from $w_N$ after the suppression $udu\to ud$ (resp. $dud\to du$).
The homeomorphism type of the pair $(|X_1|, N)$ has not been altered by this procedure.
Finally, as explained at the end of subsection \ref{SSect_Simple}, we may perform a barycentric subdivision on $X_1$ and keep one half of the triangulated annulus $\Sd(N)$, namely the one containing  the subdivided boundary component encoded by $u$,  to get a new annulus whose simple triangulation gets encoded by a word deduced from  $w_N$ after one subdivision, see Figure \ref{Fig_Subdiv}.

\begin{figure}[h]
   \begin{center}
    \includegraphics[scale=0.29]{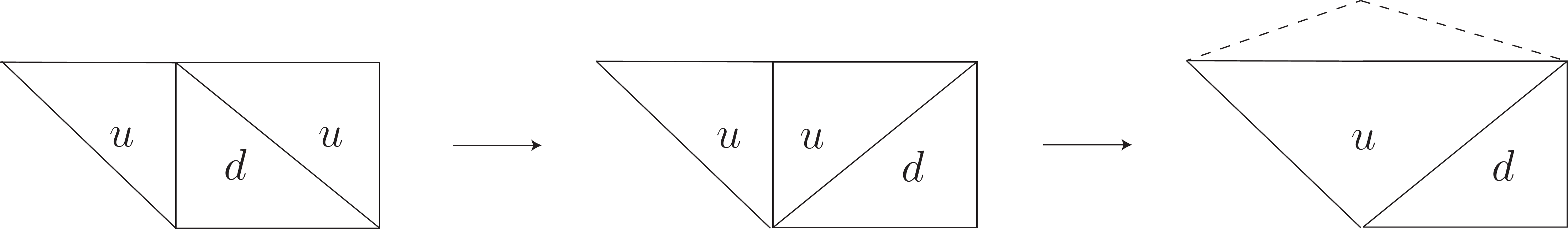}
    \caption{Suppression $udu\to ud$.}
    \label{Fig_UD_DU}
      \end{center}
 \end{figure}

By  Corollary \ref{Cor_BarSubM} and Lemma \ref{Lem_Word}, we may thus assume that the triangulation on the regular neighborhood $N$ is in fact encoded by  $w = ududdu$ and moreover  that any simplex  of $X_1$ intersects $N$ along a single face, possibly empty.
Corollary \ref{Cor_Handle} then provides a Morse shelled two-handle that can be attached to the boundary component $\Sigma$ along $N$. 
 Moreover, we may attach the three tiles of the handle one after the other to $X_1$, following the shelling order, to extend the Morse shelling of $X_1$ through the handle. 
 
 In order to prove Theorem \ref{Thm_intro2}, it remains now to check that it is also possible to attach a Morse shelled top-dimensional handle in dimensions two  and three. Let thus $X_2$ be a simplicial complex homeomorphic to an $n$-dimensional manifold with boundary and equipped with a Morse shelling and let $\Sigma$ be a boundary component of $X_2$ which we assume to be homeomorphic to a  $(n-1)$-sphere. By Corollary \ref{Cor_BarSubM} and Lemma \ref{Lem_Word} we may assume that any simplex of $X_2$ intersects $\Sigma$ along a single face, possibly empty. If $n=2$, $\Sigma$ is a triangulated circle. If the latter contains only three edges, we glue a basic tile of order one followed by a basic tile of order two and three to cap it, for the union $T_1^2\sqcup T_2^2\sqcup T^2_3$ is the shelled open disk $\partial \Delta_3\setminus \Delta_2.$ If $\Sigma$ contains more than three edges, we may find two consecutive ones which are not faces of a same triangle of $X_2$ and attach a basic tile of order two along them to decrease by one the number of edges in this boundary component. After a finite induction, we are led to the previous case and cap $\Sigma$ to get a Morse shelled simplicial complex homeomorphic to  a manifold with one less boundary component. Theorem \ref{Thm_intro2} is now proved in dimension $n=2$.
 
 If $n=3$, $\Sigma$ is a triangulated two-sphere. If this triangulation has just four vertices, we attach a basic tile of order one followed by basic tiles of order two, three and four to cap $\Sigma$, for the union $T_1^3\sqcup T_2^3\sqcup T^3_3\sqcup T^3_4$ is the shelled open disk $\partial \Delta_4\setminus \Delta_3.$ If this triangulation has more than four vertices, we are going to prove by induction that  it can be modified to reduce its number of vertices.    Indeed, if $\Sigma$ contains a vertex $v$ of valence three, there is no obstruction to attach a basic tile of order three along the three adjacent triangles  to get a new  simplicial complex  and triangulated sphere with the same vertices  as $\Sigma$ but $v$, see Figure \ref{Fig_VBary}, since the triangulation has more than four vertices by hypothesis.

  \begin{figure}[h]
   \begin{center}
    \includegraphics[scale=0.35]{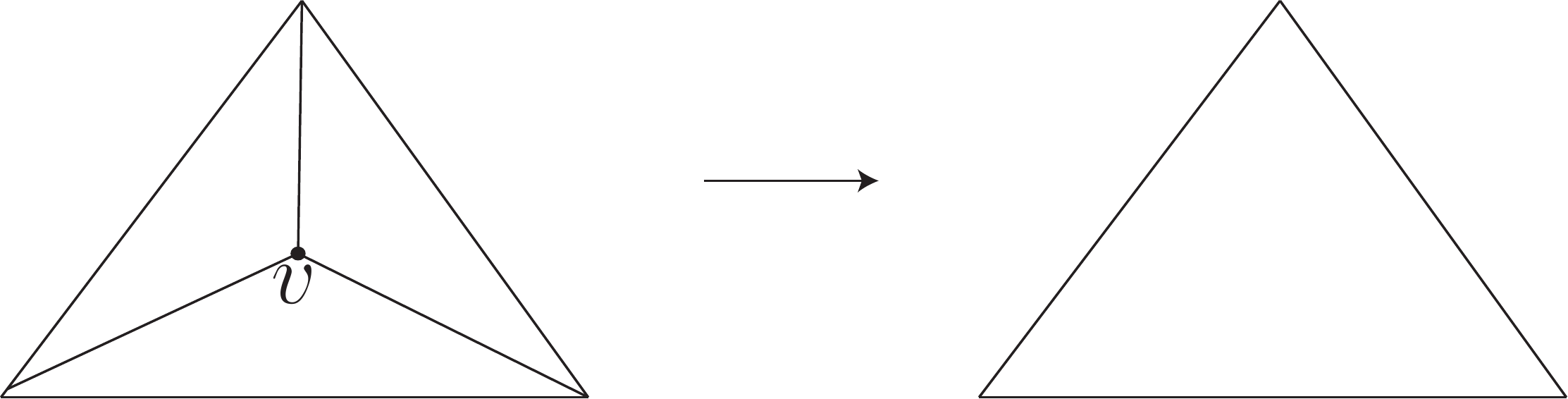}
    \caption{Removal of a vertex.}
    \label{Fig_VBary}
      \end{center}
 \end{figure}

    If $\Sigma$  contains a vertex $v$ of valence greater than three, then its link $\Lk_\Sigma(v)$ in $\Sigma$ is a triangulated circle containing more than three vertices. From Jordan's theorem we deduce that at least one of these vertices, say $w$, has the property that 
   if an edge in $\Sigma$ has its boundary in $\Lk_\Sigma(v)$ and contains $w$, then this edge is  included in     $\Lk_\Sigma(v)$. We then choose two consecutive triangles adjacent to $v$, one of them having $[v,w]$ in its boundary, and attach a basic tile of order two along them.   We get a new simplicial complex and a triangulated sphere with same vertices, but the valence of $v$ has decreased by one, see Figure \ref{Fig_VDiag}.

\begin{figure}[h]
   \begin{center}
    \includegraphics[scale=0.27]{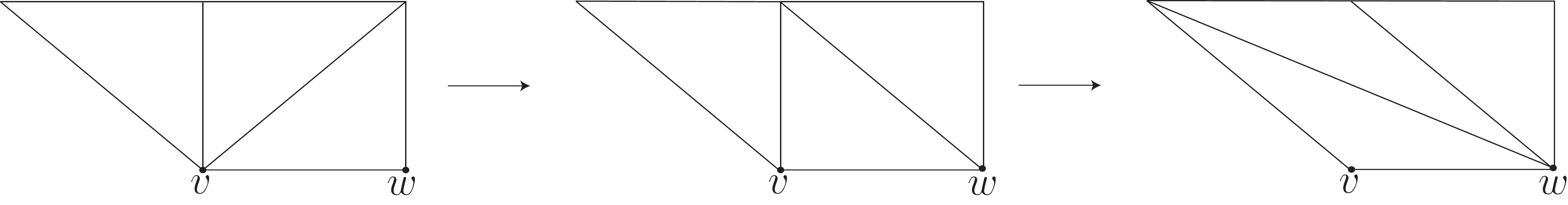}
    \caption{Decreasing the valence of a vertex.}
    \label{Fig_VDiag}
      \end{center}
       \end{figure}
       
     Moreover, the basic tile of order two has still one triangle adjacent to $v$ and $w$. We attach a basic tile of order two along this triangle and the one   next to it adjacent to $v$ not containing $w$ to decrease the valence of $v$ by one more, as in Figure \ref{Fig_VDiag}. We may continue by induction until the valence of $v$ has decreased to three and then attach as before a basic tile or order three to get a new triangulated sphere with same vertices as $\Sigma$ but $v$. Moreover, once this last attachement is performed, we keep the property that any simplex of the simplicial complex intersects its boundary along a single  face, possible empty. After a finite induction, we are led to the case where $\Sigma$ is a triangulated sphere with just four vertices and cap it  as before by a shelled open disk. We have thus attached a triangulated three-handle to $X_2$ along $\Sigma$ and extended the shelling of $X_2$ through this handle. Hence the result.
 \end{proof}
 
 The shellings given by the proof of Theorem \ref{Thm_intro2} do not use any regular Morse tile of vanishing order, as the ones given by the proof of Theorem \ref{Thm_intro1}.

\begin{proof}[Proof of Corollary \ref{Cor_intro}]
Let $f$ be a discrete Morse function compatible with $\mathcal{T}$.  By Theorem \ref{Thm_DiscreteMorse}, the critical points of $f$ are in one-to-one  correspondence, preserving the index, with the critical tiles of the tiling, so that for every $k \in \{ 0, \dots , n \}$, $f$ has $c_k (\mathcal{T}) $ critical points of index $k$. Theorem~$7.3$ of \cite{F1} then provides a chain complex which computes the homology of $X$ and which has dimension $c_k (\mathcal{T}) $ in grading $k$, it is the discrete Morse complex. The result then follows from the classical Morse inequalities
deduced from this chain complex. 
\end{proof}

\subsection{Final remarks}\label{SSect_Remarks}

\begin{enumerate}

\item It would be of interest to prove Theorem \ref{Thm_intro2} in any dimension. In  \cite{Welsch}, we deduce from Theorem \ref{Thm_intro2} that any finite product of closed manifolds of dimension less than four carries Morse shellable triangulations. By \cite{Welsch3}, every finite simplicial complex becomes Morse shellable after finitely many stellar subdivisions at maximal simplices. 

\item The $h$-tiling of $\partial \Delta_2$ made of three basic tiles of order one has no critical tile. Example $4.5$ of \cite{SW3} also provides h-tiled triangulations on the two-torus having no critical tile, while by Theorem $1.1$ of \cite{Welsch}, any product of a sphere and a torus of positive dimension carries $h$-tileable triangulations using only regular tiles. In these examples, every discrete vector field compatible with the tiling has closed non-stationary path, so that by Theorem \ref{Thm_9.3} they cannot be the gradient vector fields of some discrete Morse functions.

\item  By Lemma \ref{Lem_chitile} and the additivity of the Euler characteristic, an even dimensional closed manifold equipped with an $h$-tiled triangulation has non-negative Euler characteristic. 
In particular, no triangulation on a closed surface of genus greater than one is $h$-tileable. We do not know which closed three-manifold carry $h$-tileable triangulations.

\item The two-dimensional simplicial complex containing four triangles depicted in Figure \ref{fig_triangles}, where the points $a$ are glued together, is not Morse tileable. It would be of interest to exhibit a closed triangulated  manifold which is not Morse tileable.

 \begin{figure}[h]
   \begin{center}
    \includegraphics[scale=0.3]{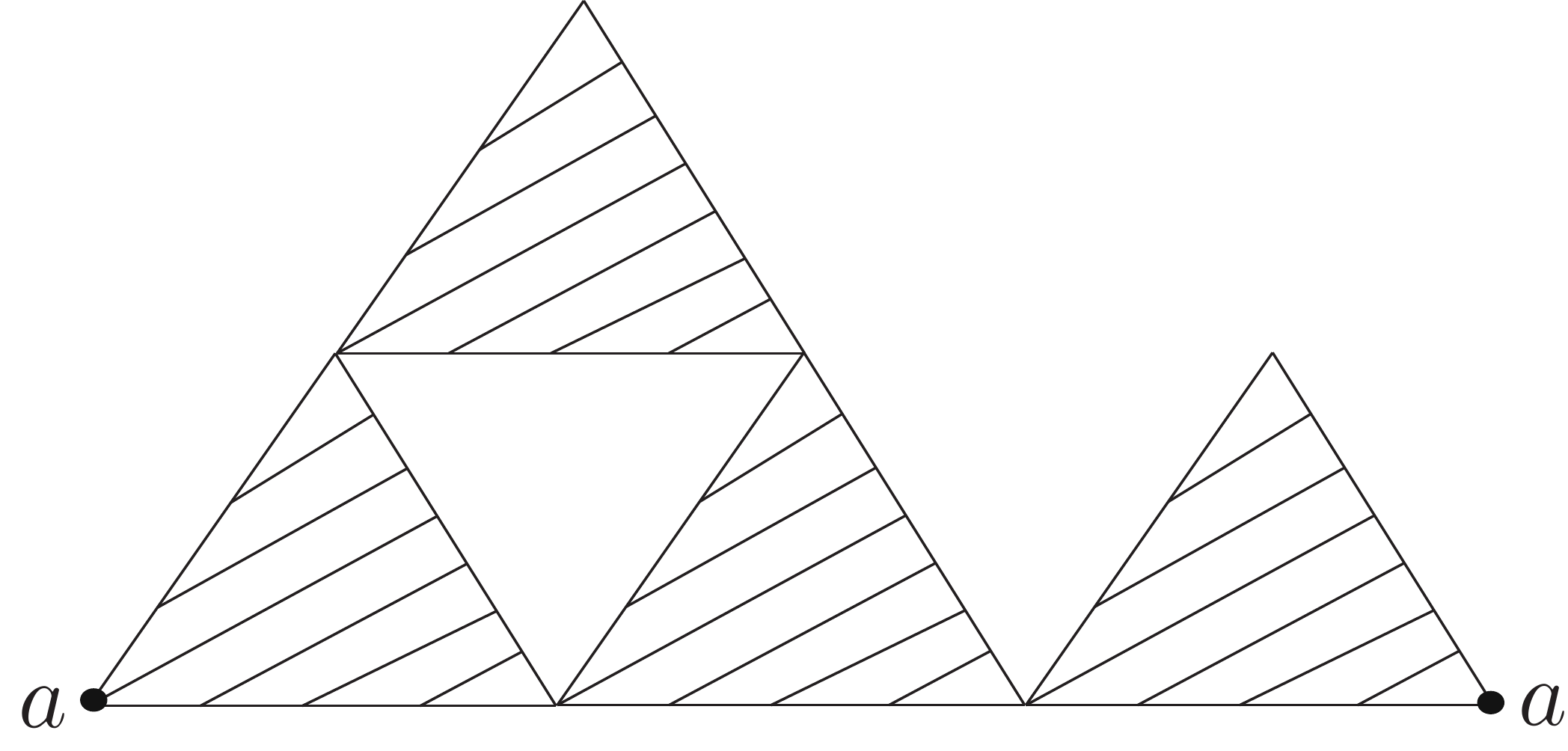}
    \caption{A simplicial complex which is not Morse tileable.}
    \label{fig_triangles}
      \end{center}
 \end{figure}
 
 \item Triangulations on any smooth closed manifold always exist, see for example \cite{W}, and topological closed manifolds of dimension less than four have a unique smooth structure, see \cite{Bing, Mo}.  Also, Morse functions exist on all closed manifolds, see~\cite{Milnor2} and handle decompositions also exist  on $PL$-manifolds, see Proposition~6.9 of~\cite{RS}.

\end{enumerate}

Univ Lyon, Universit\'e Claude Bernard Lyon 1, CNRS UMR 5208, Institut Camille Jordan, 43 blvd. du 11 novembre 1918, F-69622 Villeurbanne cedex, France

{salepci@math.univ-lyon1.fr, welschinger@math.univ-lyon1.fr.}
\end{document}